\documentclass{amsart}

\usepackage[a4paper,top=3cm,bottom=3cm,left=3cm,right=3cm]{geometry}
\usepackage{amsmath}
\usepackage{amssymb}
\usepackage{paralist}
\usepackage{graphics} %% add this and next lines if pictures should be in esp format
\usepackage{epsfig} %For pictures: screened artwork should be set up with an 85 or 100 line screen
\usepackage{graphicx}  \usepackage{epstopdf}%This is to transfer .eps figure to .pdf figure; please compile your paper using PDFLeTex or PDFTeXify.
\usepackage[colorlinks=true]{hyperref}
\usepackage{xcolor}

\hypersetup{urlcolor=blue, citecolor=blue}
\newtheorem{theorem}{Theorem}[section]

\newtheorem{lemma}[theorem]{Lemma}
\newtheorem{proposition}{Proposition}

\theoremstyle{definition}
\newtheorem{definition}[theorem]{Definition}
\newtheorem{remark}{Remark}

\newcommand{\tc }{\mathtt{c}}

\newcommand{\R}{\mathbb{R}}
\newcommand{\T}{\mathbb{T}}
\newcommand{\C}{\mathbb{C}}

\newcommand{\Z}{\mathbb{Z}}
\newcommand{\N}{\mathbb{N}}

\newcommand{\bral}{[\![}
\newcommand{\brar}{]\!]}

\newcommand{\cK}{\mathcal{K}}

\title[] %Use the shortened version of the full title
{Sobolev instability in the cubic NLS equation with convolution potentials on irrational tori}
\author[Filippo Giuliani]{Filippo Giuliani}

\thanks{The author has received funding from INdAM-GNAMPA, Project CUP E55F22 000270001. }
\begin{document}
	
	\maketitle
	
	\vspace{-1cm}
	
	% Enter the first author's name and address:
	%\centerline{\scshape Filippo Giuliani$^*$}
	\medskip
	{\footnotesize
		% please put the address of the first author
		 \centerline{Dipartimento di matematica, Politecnico di Milano, }
		   \centerline{Piazza Leonardo Da Vinci 32, 20133, Milano, Italy}
		    \centerline{Email address: filippo.giuliani@polimi.it}
	} % Do not forget to end the {\footnotesize by the sign }

	\begin{abstract}  
		In this paper we prove the existence of solutions to the cubic NLS equation with convolution potentials on two dimensional irrational tori undergoing an arbitrarily large growth of Sobolev norms as time evolves. Our results apply also to the case of square (and rational) tori.
		We weaken the regularity assumptions on the convolution potentials, required in a previous work by Guardia \cite{Guardia14} for the square case, to obtain the $H^s$-instability ($s>1$) of the elliptic equilibrium $u=0$.
		We also provide the existence of solutions $u(t)$ with arbitrarily small $L^2$ norm which achieve a prescribed growth, say $\|u(T)\|_{H^s}\geq \cK \| u(0) \|_{H^s}$, $\cK \gg 1$, within a time $T$ satisfying polynomial estimates, namely $0<T<\cK^c$ for some $c>0$. 
	\end{abstract}

	\tableofcontents

\section{Introduction}

This paper concerns the existence of spatially quasi-periodic solutions of the cubic defocusing nonlinear Schr\"odinger equation with smooth, time-independent, convolution potentials
\begin{equation*}\label{defNLSpot}
-\mathrm{i} \partial_t u=-\Delta u+W*u+|u|^2 u, \qquad u=u(t, x),
\end{equation*}
 which undergo an arbitrarily large growth in time of their high order Sobolev norms.
 
In recent years the study of the growth of Sobolev norms for solutions of nonlinear Hamiltonian partial differential equations (PDEs) on compact manifolds has drawn wide attention in the mathematical community.
In \cite{Bourgain00b} Bourgain highlights the importance of this problem and conjectures the existence of solutions to the defocusing cubic NLS
\begin{equation}\label{defNLS}
-\mathrm{i} \partial_t u=-\Delta u+|u|^2 u,
\end{equation}
 on the two dimensional torus $\T^2=(\R/2\pi \Z)^2$
 which exhibit an unbounded growth of their $H^s$, $s>1$, Sobolev norms as time goes to infinity. We recall that the $H^1$ norm of solutions to the defocusing NLS is controlled by the energy Hamiltonian for all time. For this reason, a growth in higher order Sobolev norms is necessarily due to a migration of energy from low to high Fourier modes. This phenomenon is usually called \emph{forward energy cascade} and it is central in the context of the weak wave turbulence theory.
 
 The first results of existence of solutions exhibiting growth of Sobolev norms for PDEs on compact spatial domains dates back to the $90$'s. They are due to Bourgain \cite{Bou95,Bou96} for NLS and NLW equations with a spectrally defined Laplacian operator, and Kuksin \cite{Kuksin97b} for a NLS equation with small dispersion. 
 
Nowadays Bourgain's question \cite{Bourgain00b} remains wide open, although results of existence of unbounded solutions to NLS equations have been provided by Hani-Pausader-Visciglia-Tzvetkov \cite{HaniPTV15} on $\R\times \T^2$ and by Hani \cite{Hani12} for a cubic almost-polynomial NLS on $\T^2$. For other models rather than NLS, we refer to \cite{GerardG10}, for an existence result of unbounded solutions of the Szeg\"o equation on $\T$, {and \cite{GGlattice}, for a construction of unbounded transfers of energy orbits in an infinite pendulum lattice by means of Arnold diffusion techniques}. 

In the seminal work \cite{CKSTT} Colliander, Keel, Staffilani, Takaoka and Tao prove the existence of solutions of the defocusing cubic NLS \eqref{defNLS} on $\T^2$ undergoing an arbitrarily large (but finite) growth of their Sobolev norms on a finite range of time. More precisely they prove the following result.
\begin{theorem}{(\cite{CKSTT})}\label{thm:CKSTT}
Let $s>1$ and fix $\cK>0$ large enough, $\mu>0$ small enough. Then there exists a solution $u(t)$ of \eqref{defNLS} on $\T^2$ such that
\begin{equation}\label{pico}
\| u(0)\|_{H^s}\le \mu, \qquad \| u(T)\|_{H^s}\geq \cK
\end{equation}
for some time $T>0$.
\end{theorem} 
The above result does not answer to Bourgain's conjecture. However, {it establishes the existence} of an interesting dynamical phenomenon. Indeed Theorem \ref{thm:CKSTT} can be seen as a result of strong Lyapunov instability of the origin $u=0$ in the topology of Sobolev spaces. We refer to this property as \emph{long-time instability} of the elliptic equilibrium $u=0$.

In \cite{CKSTT} the authors develop a mechanism to construct unstable solutions with a dynamical system approach. This is based on the derivation of a resonant finite dimensional system, called \emph{toy model}, which approximates the dynamics of the NLS equation over a long time interval and possesses orbits exhibiting a forward energy cascade behavior: at initial time the energy is mostly localized on low modes and later it is transferred to high modes. By a Gronwall argument they prove that there exist solutions to NLS equation which stay close, in an appropriate weak norm, to the energy cascade orbits of the toy model. The vicinity in this weak norm is sufficient to show that these NLS solutions undergo the prescribed growth in their Sobolev norms.

An interesting question concerns bounds for the time at which these solutions achieve the prescribed growth (the time $T$ in Theorem \ref{thm:CKSTT}). Regarding this,
Guardia-Kaloshin \cite{GuardiaK12, GuardiaK12Err} refine the mechanism developed in \cite{CKSTT} and prove that long-time unstable solutions satisfying \eqref{pico} can be constructed in such a way that the time $T$ has a super-exponential upper bound in terms of $\cK/\mu$, that is
\[
0<T \le \exp\left({\left({\cK}/{\mu}\right)^c}\right) \qquad \mathrm{for\,\,some\,\,}c>0.
\] They also show that, if one does not assume smallness on the $H^s$ norm of the initial datum, then one can construct solutions $u(t)$ exhibiting an arbitrarily large growth, say
\[
\| u(T) \|_{H^s}\geq \cK \| u(0)\|_{H^s} \quad  \mathrm{for\,\, an\,\, arbitrarily\,\, large}\,\,\cK>0,
\]
with the time $T$ satisfying a polynomial bound in $\cK$, namely
\[
0<T\le \cK^c\qquad \mathrm{for\,\,some\,\,}c>0.
\]
Moreover, in this case, the solutions can be chosen to have an arbitrarily small $L^2$ norm.

In \cite{Guardia14} Guardia shows that the construction mechanism designed in the aforementioned papers works also for the cubic NLS equation with convolution potentials on $\T^2$. In particular, he shows that the $H^s$ long-time instability property still holds in presence of convolution potentials belonging to $H^{s_0}$ with $s_0>(70/17) s$.

 After these works, many efforts have been made to show that the phenomenon of arbitrarily large growth in Sobolev norms holds for other NLS equations. Haus-Procesi \cite{HPquintic} provide Sobolev instability of the origin for the quintic NLS equation. Guardia-Haus-Procesi \cite{GuardiaHP16} extend the previous results to any NLS with analytic nonlinearity. We mention also \cite{Hani12}, \cite{GuardiaHHMP19} where the phenomenon of the $H^s$ long-time instability, for $s\in (0, 1)$, is proved for different invariant objects (rather than fixed points) of the cubic NLS equation, respectively plane waves and finite-gap quasi-periodic solutions. 
The problem of proving the $H^s$ instability of such invariant objects for $s>1$ remains an interesting open question.

\smallskip

All the aforementioned results concern NLS equations on square tori. 
As highlighted by Staffilani-Wilson \cite{StafWil} (see also \cite{StafWil2}), in the case of irrational tori, that is
\[
\T^2_{(1, \omega)} :=\T_1\times \T_{\omega}, \qquad \mathrm{with}\quad \,\omega\notin\mathbb{Q}, \qquad \mathrm{and} \qquad \T_{\lambda}:=\R/ (2\pi \lambda^{-1} \Z), \,\lambda>0,
\]

 the resonances decouple into products of one dimensional resonances and the mechanism designed in \cite{CKSTT} cannot be applied in a straightforward way. Indeed the derivation of the toy model in \cite{CKSTT}  is based on the presence of non-degenerate $4$-wave resonances, which are not present in the irrational case. To better understand the next arguments, it is useful to go into more detail. The toy model arises as the restriction of the first order resonant NLS Hamiltonian to a finite dimensional subspace $\mathcal{V}_{\Lambda}$, which is defined by setting at zero all the modes which are not indexed by a particular subset $\Lambda\subset \Z^2$. The construction of the set $\Lambda$ is one of the main ingredient of the proof of Theorem \ref{thm:CKSTT}. Indeed, this set possesses several combinatorial properties which makes simpler the study of the resonant toy model. Roughly speaking, in the construction of the set $\Lambda$ the building blocks are the non-degenerate $4$-wave resonances, which can be visualized as the vertices of a rectangle in the $\Z^2$-lattice. However, in the irrational case, the only non-degenerate $4$-wave resonances form rectangles with sides parallel to the axes and these ones are not enough to build the set $\Lambda$ designed in \cite{CKSTT}.
 
Recently, the author of the present paper together with Guardia \cite{GG21} proved the existence of solutions to the cubic NLS on two dimensional irrational tori $\T^2_{(1, \omega)}$ which display: 
\begin{itemize}
\item[(i)] an arbitrarily large growth in Sobolev norms, without assuming smallness in the $H^s$ norm of the initial data, for almost all choices of the irrational spatial frequency $\omega$; 
\item[(ii)] the long-time unstable behavior described in Theorem \ref{thm:CKSTT}, assuming more strict conditions on $\omega$.
\end{itemize}

One of the key arguments of the strategy adopted in \cite{GG21} is to scale the $\Z^2$-lattice in such a way that the set $\Lambda$ of \cite{CKSTT} is mapped into a set with similar combinatorial properties and that is composed by \emph{quasi-resonances}. The choice of the scaling is based on certain results of diophantine approximation theory.  The toy model obtained in this way has the same dynamics of the one introduced in \cite{CKSTT}, but it is a quasi-resonant system. Then one issue is to quantify how far is this system to be resonant. This is needed to ensure that the dynamics of the toy model still approximates the dynamics of NLS over a certain range of time.

%This means that the dynamics of the toy model approximates worse the dynamics of NLS.  
Another fundamental point is to perform a partial version of the Birkhoff normal form procedure. Indeed, on the irrational torus the linear frequencies of oscillations are not integers (as in the square case) and small divisors issues arise when one wants to fully normalize the Hamiltonian of degree four. However, in \cite{GG21} it is proved that it is not necessary to eliminate/normalize all the Hamiltonian terms, but just a part of them which generates a finite dimensional vector field (this procedure is also called \emph{weak} Birkhoff normal form).  This allows to avoid the accumulation to zero of the combinations of the linear frequencies which causes the small divisors problem in the construction of the normalizing change of coordinates.

% \red{This result is based on the use of quasi-resonances, a weak version of the Birkhoff normal form method and results of diophantine approximation.(DA ESTENDERE)}

\smallskip
 
The aim of the present paper is to show that the strategy developed in \cite{GG21} is robust and allows to take into account also the presence of convolution potentials, in the spirit of the work \cite{Guardia14} for the square case.

 In the case in which we assume smallness only on the $L^2$ norm of the initial datum we improve the estimate on the time provided in \cite{GG21} at which the prescribed growth is achieved. In particular, we recover the polynomial time estimates of \cite{GuardiaK12} obtained for the square case. This fact gains importance given that the growth of Sobolev norms is expected to be \emph{weaker} on irrational tori (we refer to \cite{Deng, DengGerm, StafWil}). It would be also interesting to understand how the existence of such solutions, exhibiting a relative fast growth, relates to the results in \cite{StafWil2}, where the time of spreading of energy at high frequencies is given for arbitrarily large initial data with compact Fourier support.

We also analyze how the upper bound on the time depends on $\omega$ when we look at the limit $\omega\to +\infty$, which can be seen, heuristically, as the transition from the NLS problem on a $2d$ rectangular torus to the periodic, completely integrable, $1d$ case.

{The proof of the main results (Theorems \ref{thm:weak}, \ref{thm:strong} below) can be easily adapted\footnote{Along the proof we make some remarks to highlight the major differences between the irrational and  square case. } to the square (and even rational) case to give an alternative proof of the result by Guardia \cite{Guardia14}.

We point out that Theorem \ref{thm:strong} provides the $H^s$ instability of the origin $u=0$ by requiring weaker regularity assumptions on the convolution potentials with respect to \cite{Guardia14}.} Below we discuss how the critical regularity threshold given by Theorem \ref{thm:strong} is connected with the problem of showing the forward energy cascade phenomenon in the vicinity of plane waves and finite gap solutions.

%We think that it could be interesting to understand whether this regularity condition is optimal (or not) by comparing the main results of the present papers with the available results of long time stability for Hamiltonian PDEs on higher dimensional tori \red{cite}

\subsection{Main results}
We consider the defocusing cubic NLS equation on two dimensional irrational tori
\begin{equation}\label{NLS0}
-\mathrm{i} \partial_t u=-\Delta u+W*u+|u|^2 u \qquad u=u(t, x) , \qquad t\in \R, \qquad x=(x_1, x_2)\in \T^2_{(1, \omega)},
\end{equation}
where $W=W(x)$ is a convolution potential with real Fourier coefficients.
We investigate the existence of smooth solutions to \eqref{NLS0} which undergo an arbitrarily large growth in their Sobolev norms as time evolves. Using the Fourier series expansion
\[
u(t, x_1, x_2)=\sum_{n=(j, k)\in \Z^2} u_n(t)\,e^{\mathrm{i} (j x_1+k x_2)}, \quad u_n:=\frac{\omega}{(2\pi)^2} \int_{\T^2_{(1, \omega)}} u(x_1, x_2)\,e^{-\mathrm{i} (j x_1+k x_2)}\,dx_1 d x_2 
\]
we introduce, for $s\geq 0$, the Sobolev spaces
\[
H^s=H^s\left(\T^2_{(1, \omega)}\right)=\left\{ u\in L^2\left(\T^2_{(1, \omega)}\right) : \| u \|^2_{H^s}=\sum_{n\in  \Z^2} |u_n|^2 \langle n \rangle^{2s} <+\infty \right\}, \quad \langle n \rangle:=\max\{ 1, |n| \}.
\]
\noindent\textbf{Notation.} We use the notation $f\lesssim g$ to denote $f\le C g$ whenever $C>0$ is some constant independent of the parameters of the problem. We write $f\sim g$ if there exist two constants $0<C_1<C_2$ such that $C_1 |f|\le |g|\le C_2 |f|$.

\smallskip

We are in position to state the main results of the paper. The first result concerns the existence of solutions of equation \eqref{NLS0} exhibiting an arbitrarily large growth of Sobolev norms. We consider the case of solutions with arbitrarily small $L^2$ norm, without assuming any smallness conditions on the $H^s$ norms, $s>1$, of the initial data.
\begin{theorem}\label{thm:weak}
Let $s>1$, $s_0>0$ and $W\in H^{s_0}(\T^2_{(1, \omega)})$ with real Fourier coefficients. Fix $\cK>0$ large enough. For almost all irrational numbers $\omega>1$ there exists a smooth solution $u(t)$ of the cubic NLS equation \eqref{NLS0} such that 
\begin{equation}\label{weakI}
\| u(0)\|_{L^{2}}\lesssim \cK^{-1}, \qquad \frac{\| u(T) \|_{H^s}}{ \| u(0) \|_{H^s}}\geq {\cK} \qquad \mathrm{for\,\,} 0<T\le (\cK \omega^{s})^c
\end{equation}
for some universal constant $c>0$.
%\item[(ii)]If we do not impose a restriction on the $L^2$ norm of the initial datum then we can construct solutions achieving the growth \eqref{weakI} within the time
%\begin{equation}\label{time2}
%T\lesssim \frac{(\log(\cK\,\omega^s))^2}{s-1}.
%\end{equation}
%\end{itemize}
\end{theorem}

\begin{remark}
The result given by Theorem \ref{thm:weak} for $\omega=1$ is provided by Theorem $1.4$ in \cite{Guardia14}.
\end{remark}

\begin{remark}
As in the square case, we are able to construction solutions exhibiting arbitrarily large growth for any potential which is slightly more regular than $L^2$.
We do not need to assume any smallness conditions on $W$, nor resonant conditions on its Fourier coefficients. For our purpose it is enough to have a weak decay of the Fourier coefficients. Then the growth of the $H^s$ norms is achieved for open sets of potentials.
\end{remark}

Theorem \ref{thm:weak} should be compared with Theorem $2.1$ in \cite{GG21}, where the same result is provided in absence of the potential ($W\equiv 0$).
In \cite{GG21} the bound on time $T$ is super-exponential in the factor $\cK$, while in \eqref{weakI} is polynomial. We remark that this is the same type of upper bound provided by Guardia-Kaloshin \cite{GuardiaK12} in the square case.
The key argument used here to obtain the improved bound relies on the normal form procedure described in Section \ref{sec:NF}. We construct a change of coordinates which eliminates part of the NLS Hamiltonian, simplifying the analysis of its first order close to the origin. The drawback is that new higher order terms appear in the Hamiltonian. We refer to them as the remainder of the normal form procedure (it is called $\mathcal{R}$ in Proposition \ref{prop:wbnf}). The size of the remainder affects the time $T$.
We shall fully exploit the fact that the aforementioned change of coordinates is identity modulo a smoothing term\footnote{This is a consequence of having good lower bounds on certain combinations of linear frequencies (see Lemma \ref{lem:L1}).}. The smoothing effect is used to gain some smallness on the size of the remainder (see the bound \eqref{bound:remainder}), ensuring a better bound on $T$. 
 
 % we have good lower bounds on the resonant $4$-wave interactions involved in the partial Birkhoff normal form procedure of Proposition \ref{prop:wbnf}. 

It is natural to expect a blow up of the time $T$ in \eqref{weakI} as the ratio of the spatial frequencies tends to infinity (that is $\omega\to+\infty$). Heuristically, this limit should resemble the situation of the NLS equation in dimension one where the phenomenon of growth of Sobolev norms is prevented by the complete integrability.  The above theorem shows that the time $T$ grows at most polynomially with respect to $\omega$. 

\medskip

As commented above, for the solutions provided in Theorem \ref{thm:weak} we do not assume the smallness of the $H^s$ norms with $s>1$ at initial time.  Assuming more strict conditions on the spatial frequency $\omega$ and on the regularity of the convolution potentials $W$, we can ensure the existence of solutions of \eqref{NLS0} exhibiting the long-time instability of $u=0$, as in Theorem \ref{thm:CKSTT}.
\begin{theorem}\label{thm:strong}
Let $s>1$, $s_0>2 s$ and $W\in H^{s_0}(\T^2_{(1, \omega)})$ with real Fourier coefficients. There exists a set $\widetilde{\mathcal{S}}_{2s}\subset[1, +\infty)$, which has  Hausdorff dimension $(1+s)^{-1}$, such that for all $\omega\in \widetilde{\mathcal{S}}_{2s}$ the following holds.

Fix $\cK>0$ large enough and $\mu>0$ small enough. There exists a smooth solution $u(t)$ of the cubic NLS equation \eqref{NLS0} such that 
\begin{equation}\label{bound:uppermu}
\| u(0)\|_{H^s}\le \mu, \qquad \| u(T)\|_{H^s}\geq \mathcal{K}.
\end{equation}
Moreover, for any $\tau>2s$, there exists a subset $\widetilde{\mathcal{S}}_\tau\subset \widetilde{\mathcal{S}}_{2s}$ with Hausdorff dimension $2(2+\tau)^{-1}$, such that if $\omega \in \widetilde{\mathcal{S}}_\tau$, the time $T$ satisfies
\begin{equation}\label{time2}
% \exp\left(\,\frac{1}{\tau-s}\left(\frac{\mathcal{C}}{\mu}\right)^{\beta_1}\right) \le 
0<T\le\exp\left(\,\frac{1}{\tau-2s}\left(\frac{\mathcal{K}\,\omega^s}{\mu}\right)^{\beta}\right) 
\end{equation}
for some constants $\beta=\beta(s)>0$.
\end{theorem}

\begin{remark}
Theorem \ref{thm:strong} also includes the case $\omega=1$, hence it extends Theorem $1.5$ in \cite{Guardia14}, where the potentials are required to be in $H^{s_0}$ with $s_0>(70/17)s$.
\end{remark}
\begin{remark}
{With respect to Theorem \ref{thm:weak}, we are just able to prove that the time $T$ grows super-exponentially with respect to $\omega$.} 
\end{remark}

It is not clear whether the regularity required in Theorem \ref{thm:strong} on the convolution potential is optimal or not. However, it is interesting to notice that the difficulty in improving this threshold seems related to the (unsolved) problem of proving $H^s$ instability, for $s>1$, of plane waves \cite{Hani12} and finite gap solutions \cite{GuardiaHHMP19} of the cubic NLS on $\T^2$ (at least if one tries to follow the strategy of \cite{CKSTT}). Indeed, in such cases the asymptotic of the linear frequencies of oscillations is roughly given by 
\[
\Omega(n)=|n|^2+O\left(\frac{1}{| n|^2}\right),
\] which could be seen, at first orders, as the eigenvalues of $-\Delta+W*$ with a convolution potential $W\in H^2$. This is in fact the critical case for Theorem \ref{thm:strong}.

The sets of frequencies $\widetilde{S}_{\tau}$ in Theorem \ref{thm:strong} are the same considered in \cite{GG21}. To describe them explicitly we need some preliminary definitions.
\begin{definition}\label{def:psiapprox}
Let $\psi\colon \N\to \R^+:=[0, \infty)$ be a decreasing function. We say that $\omega\in \R$ is $\psi$-approximable if there exist infinitely many $(p, q)\in \Z\times\N$ such that
\[
\left|\omega-\frac{p}{q} \right|\le \frac{\psi(q)}{q}.
\]
We say that $(p, q)$ is a $\psi$-convergent of $\omega$.
\end{definition}

\begin{definition}\label{def:Wpsi}
Let $\psi\colon \N \to \R^+$ be a decreasing function.
Let us define
\[
W(\psi)=\left \{ \omega\in [1, \infty) : \omega\,\,\mathrm{\,\,is}\,\,\psi\,-\,\text{approximable}\right \}.
\]
\end{definition}

The set $\widetilde{S}_{\tau}$ is the set of frequencies $\omega\geq 1$ that are $\psi$-approximable with
\[
\psi(q)=\frac{\mathtt{c}}{q^{1+\tau}}
\]
for some constant $\mathtt{c}\geq 1$. By Khinchin theorem \cite{Kinch} these sets have zero Lebesgue measure. The Hausdorff dimension is provided by the Jarn\'ik-Besicovitch Theorem \cite{Besi, Jarnik}.

\subsection{Plan of the paper.} Section \ref{sec:proof} is devoted to the proof of Theorems \ref{thm:weak}, \ref{thm:strong}. For technical convenience we first manipulate equation \eqref{NLS0} and we provide a suitable Hamiltonian setting to work with.
 In Section \ref{sec:initialdatum} we show how to construct the Fourier support of the initial data of solutions exhibiting the Sobolev norms explosion. In Section \ref{sec:NF} we apply a partial normal form procedure which allows to derive the toy model and to obtain the improved bound \eqref{weakI} on the time $T$ in Theorem \ref{thm:weak}. In Section \ref{sec:resmodel} we derive the finite dimensional toy model, as a quasi-resonant subsystem of the normalized Hamiltonian, and we provide the existence of forward energy cascade orbits for it. The first difference in the proof of the two main results (Theorems \ref{thm:weak}, \ref{thm:strong} ) comes with the \emph{approximation argument} (that is given in Propositions \ref{prop:approxarg} and \ref{prop:approxarg2}), namely when we have to show that, under appropriate assumptions, there exist solutions of the NLS equation which stay close to the orbits of the toy model constructed in Section \ref{sec:resmodel}. Since Proposition \ref{prop:approxarg2} follows closely the approach used in \cite{GG21}, we first prove in full details Proposition \ref{prop:approxarg} in Section \ref{sec:approxarg} and we conclude the proof of Theorem \ref{thm:weak} in Section \ref{sec:conclusion}. Eventually, in Section \ref{sec:thm2}, we conclude the proof of Theorem \ref{thm:strong}.

\textbf{Acknowledgments} The author wishes to thank M. Procesi, M. Guardia and R. Scandone for useful comments and discussions. The author has received funding from INdAM-GNAMPA, Project CUP E55F22 000270001.

\section{Proof of main results}\label{sec:proof}
By scaling the spatial variable $x\in \T^2_{(1, \omega)}\rightsquigarrow (1, \omega)\cdot x\in \T^2$ the problem of finding solutions of the equation \eqref{NLS0} with spatial frequency vector $(1, \omega)$ reduces to the search for solutions of the following equation on the square torus (with abuse of notation we rename $x$ the scaled variable)
\begin{equation}\label{NLS}
-\mathrm{i} \partial_t u=-\Delta_{\omega} u+V*u+|u|^2 u, \qquad x=(x_1, x_2)\in \T^2,
\end{equation}
where $\Delta_{\omega}$ is the Laplacian operator defined by
\[
\Delta_{\omega} \,e^{\mathrm{i} n \cdot x}=|n|_{\omega}^2 \,e^{\mathrm{i} n \cdot x}, \qquad |n|_{\omega}^2=j^2+\omega^2 k^2 \qquad \forall n=(j, k)\in \Z^2
\]
and $V$ is the scaled potential $V(x_1, x_2)=W(x_1, \omega^{-1} x_2)$. 
Let us consider the Fourier series
\[
u(t, x)=\sum_{n\in \Z^2} u_n(t)\,e^{\mathrm{i} n \cdot x}, \qquad u_n:=\frac{1}{2\pi} \int_{\T^2} u(x)\,e^{-\mathrm{i} n \cdot x}\,dx.
\]

The equation \eqref{NLS} possesses a Hamiltonian structure given by $H=H^{(2)}+H^{(4)}$ with
\begin{align*}
H^{(2)}&= \sum_{n\in \Z^2} (|n|^2_{\omega}+V_n) \,|u_n|^2, \qquad H^{(4)}=\frac{1}{2} \sum_{n_1-n_2+n_3-n_4=0} u_{n_1} \overline{u_{n_2}} u_{n_3} \overline{u_{n_4}},
\end{align*}
associated with the symplectic form $\mathrm{i} \sum_{n\in \Z^2} d u_n\wedge d \overline{u_n}$. 

\smallskip

\noindent\textbf{Notation.} We shall denote by $X_H$ and $\Phi_H^t$ respectively the vector field and the flow generated by a Hamiltonian $H$. We denote by $D X_H$ the linearized vector field of $H$.

\smallskip

The $L^2$ norm of solutions to the NLS equation \eqref{NLS} is preserved, or equivalently we have that the mass Hamiltonian
\[
\mathcal{M}=\sum_{n\in \Z^2} |u_n|^2
\]
Poisson commutes with the energy Hamiltonian $H$. We observe that
\[
\sum_{n_1-n_2+n_3-n_4=0} u_{n_1} \overline{u_{n_2}} u_{n_3} \overline{u_{n_4}}=\sum_{\substack{n_1-n_2+n_3-n_4=0\\ n_1\neq n_2, n_4}} u_{n_1} \overline{u_{n_2}} u_{n_3} \overline{u_{n_4}}+2 \sum_{n\neq m} |u_{n}|^2 |u_m|^2+\sum_{n\in \Z^2} |u_n|^4.
\]
Then
\[
H^{(4)}-\mathcal{M}^2=-\frac{1}{2} \sum_{n\in \Z^2} |u_n|^4+\frac{1}{2}\sum_{\substack{n_1-n_2+n_3-n_4=0\\ n_1\neq n_2, n_4}} u_{n_1} \overline{u_{n_2}} u_{n_3} \overline{u_{n_4}}.
\]
Since $H$ and $\mathcal{M}^2$ commutes, if $u(t, x)$ is a solution of the Hamiltonian system given by $H$, $z=\Phi_{-\mathcal{M}^2}^t(u)=e^{-2 \mathrm{i} \mathcal{M} t}u$ is a solution of the Hamiltonian system given by
\begin{equation}\begin{aligned}\label{calH}
&\mathcal{H}=\mathcal{H}^{(2)}+\mathcal{H}^{(4)},\\
&\mathcal{H}^{(2)}= \sum_{n\in \Z^2} (|n|^2_{\omega}+V_n) \,|z_n|^2, \qquad \mathcal{H}^{(4)}=-\frac{1}{2} \sum_{n\in \Z^2} |z_n|^4+\frac{1}{2}\sum_{\substack{n_1-n_2+n_3-n_4=0\\ n_1\neq n_2, n_4}} z_{n_1} \overline{z_{n_2}} z_{n_3} \overline{z_{n_4}}.
\end{aligned}
\end{equation}
From now on we shall work with sequence spaces.
Consider the Sobolev spaces
	\[
	{H}^s=\left\{ z : \Z^2 \to \mathbb{C} : \| z \|_s^2= \sum_{n\in \Z^2} |z_n|^2 \langle n \rangle^{2 s}<+\infty \right\} \qquad 
	\mathrm{with}
\qquad
	\langle n \rangle:=\max\{ 1, |n|\},
	\]
	and the sequence space
	\[
	\ell^1=\left\{ z : \Z^2 \to \mathbb{C} : \| z \|_{\ell^1}= \sum_{n\in \Z^2} |z_n|<+\infty \right\}.
	\]
We remark that $\ell^1$ is an algebra for the convolution between sequences. We shall prove the existence of solutions of the Hamiltonian system given by \eqref{calH} exhibiting a large growth in the Sobolev norms $\| \cdot \|_s$ for $s>1$, but with $\ell^1$ norm which remains relatively small all the time. This allows to use the powerful tool of perturbation theory as normal forms.

We remark that, since $V\in H^{s_0}(\T^2)$, we have the following decay property on the Fourier coefficients
\begin{equation}\label{decaypot}
|V_n|\lesssim |n|^{-s_0} \qquad \forall n\in \Z^2.
\end{equation}

	\subsection{The set $\Lambda$}\label{sec:initialdatum} 
	The solutions that we are going to construct are compactly Fourier supported at initial time. We call $\Lambda$ the finite Fourier support of the initial datum. 
	 
	In this section we discuss the construction of the set $\Lambda$. This section follows the lines of Section $4$ of \cite{GG21}, where the set $\Lambda$ is constructed by scaling appropriately the $\Lambda$-set considered in \cite{CKSTT}. Here we adopt the same strategy, but it is crucial to choose a different rescaling which takes into account the correction to the Schr\"odinger linear frequencies given by the presence of the convolution potential.

	\smallskip
	
	To make no confusion we rename the $\Lambda$-set considered in \cite{CKSTT} as $\tilde{\Lambda}$.
	
	Following \cite{CKSTT} and \cite{GuardiaHHMP19}, the set $\tilde{\Lambda}$ can be decomposed as the disjoint union of sets called \emph{generations} $\tilde{\Lambda}_i$ 
	\[
	\tilde{\Lambda}=\tilde{\Lambda}_1\cup\dots\cup \tilde{\Lambda}_N.
	\]
	The number of generations $N$ is a parameter of the problem which depends on the prescribed growth to achieve. 
	
	We say that a quartet $(\tilde{n}_1, \tilde{n}_2, \tilde{n}_3, \tilde{n}_4)$ is a \emph{nuclear family} if $\tilde{n}_1, \tilde{n}_3\in \tilde{\Lambda}_i$ and $\tilde{n}_2, \tilde{n}_4\in \tilde{\Lambda}_{i+1}$ for some $i=1, \dots, N-1$ and they form a non-degenerate rectangle in the $\Z^2$-lattice.

	We recall the properties of the set $\tilde{\Lambda}$ constructed in \cite{CKSTT, GuardiaHHMP19}:
	\begin{itemize}
		\item[$(P1)$] (Closure): If $\tilde{n}_1, \tilde{n}_2, \tilde{n}_3\in\tilde{\Lambda}$ are three vertices of a rectangle then the fourth vertex belongs to $\tilde{\Lambda}$ too.
		\item[$(P2)$] (Existence and uniqueness of spouse and children): For each $1\le i\le N-1$  and every $\tilde{n}_1\in\tilde{\Lambda}_i$ there exists a unique spouse $\tilde{n}_3\in \tilde{\Lambda}_i$ and unique (up to trivial permutations) children $\tilde{n}_2, \tilde{n}_4\in\tilde{\Lambda}_{i+1}$ such that $(\tilde{n}_1, \tilde{n}_2, \tilde{n}_3, \tilde{n}_4)$ is a nuclear family in $\tilde{\Lambda}$.
		\item[$(P3)$] (Existence and uniqueness of parents and sibling): For each $1\le i\le N-1$  and every $\tilde{n}_2\in\tilde{\Lambda}_{i+1}$ there exists a unique sibling $\tilde{n}_4\in\tilde{\Lambda}_{i+1}$ and unique (up to trivial permutations) parents $\tilde{n}_1, \tilde{n}_3\in\tilde{\Lambda}_{i}$ such that $(\tilde{n}_1, \tilde{n}_2, \tilde{n}_3, \tilde{n}_4)$ is a nuclear family in $\tilde{\Lambda}$.
		\item[$(P4)$] (Non-degeneracy): A sibling of any mode $m$ is never equal to its spouse.
		\item[$(P5)$] (Faithfulness): Apart from nuclear families, $\tilde{\Lambda}$ contains no other rectangles.
		{ \item[$(P6)$] If four points $\tilde{n}_1, \tilde{n}_2, \tilde{n}_3, \tilde{n}_4$ in $\tilde{\Lambda}$ satisfy $\tilde{n}_1-\tilde{n}_2+\tilde{n}_3-\tilde{n}_4=0$ then either the relation is trivial or such points form a family.}
	\end{itemize}
	
	The nuclear families are the building blocks of the set $\tilde{\Lambda}$ and they are non-degenerate resonant quartets for the cubic NLS on the square torus. 	
	However, on the irrational torus $\T^2_{(1, \omega)}$ they are not resonances.
	We scale the $\Z^2$-lattice with an anisotropic dilation, {which depends on} the spatial frequency $\omega$, such that the rescaled nuclear families are close enough to be resonances for the equation \eqref{NLS}. 
	We consider the linear map defined by the matrix
	\begin{equation}\label{B}
		B=\begin{pmatrix}
			p & 0\\
			0 & q
		\end{pmatrix},
	\end{equation}
	where $(p, q)\in \Z^2$ is a $\psi$-convergent of $\omega$ according to the Definition \ref{def:psiapprox}. For the moment, on the function $\psi$ we just require that
	\begin{equation}\label{cond:psi}
	n \mapsto n\, \psi(n) \quad \mathrm{is\,\,monotone\,\,decreasing}.
	\end{equation}
			\begin{remark}\label{rem:monza}
		Recalling definition \ref{def:psiapprox}, we have
	$
	2^{-1} {\omega q}\le p\le 2 \omega q.
	$
	\end{remark}
	We choose $q$ in \eqref{B} such that (recall $V\in H^{s_0}$)
%	\begin{equation}\label{M}
%		q\geq M
%	\end{equation}
%	with $M\in \mathbb{N}$ large enough and such that
	\begin{equation}\label{cond:M}
		\mathtt{L}:=\frac{\omega^2}{8} q^2-4 \sup_{n\in \Z^2} |V_n|\,\,\mathrm{is\,\,large\,\,enough}, \qquad  q^{-s_0}\,\,\mathrm{is\,\,small\,\,enough}.
	\end{equation}
	In the following we will give explicit quantitative conditions on $\mathtt{L}$ and $q$.

	{\begin{remark}{(Case $\omega=1$)}
	In the square case one has just to consider an isotropic scaling, namely consider $B=q \mathrm{I}$, where $\mathrm{I}$ is the $2\times 2$ identity matrix. 
	\end{remark}}
	
	The matrix $B$ defines a scaling of the lattice
	\begin{equation*}\label{scalingLattice}
		n=(j, k)\in \Z^2 \to B\,n=(p j, q k) \in p\Z\times q\Z.
	\end{equation*} 
	For each mode $\tilde{n}_i\in \tilde{\Lambda}$ we define $n_i=B \tilde{n}_i$. Since $B$ is invertible we have a one-to-one correspondence between the generations $\tilde{\Lambda}_i\subset \Z^2$ and the image of such sets through $B$, that we denote by $\Lambda_i\subset p\Z \times q \Z$. We define
	\[
	\Lambda:=\Lambda_1 \cup \dots \cup \Lambda_N.
	\]
	The image of the nuclear families through $B$ are $(p, q)$-families according to the following definition.
	
	\begin{definition}\label{def:pqfamily}
		A  $(p, q)$-\emph{family} is a non-degenerate parallelogram with vertices $n_1, n_2, n_3, n_4\in\Lambda$ which satisfy
		\begin{equation}\label{respq}
			\begin{aligned}
				&n_1-n_2+n_3-n_4=0,\\
				&\langle B^{-2} (n_1-n_2), n_2-n_3\rangle=0.
			\end{aligned}
		\end{equation}
	\end{definition}
	
	It is easy to see that we can replace the properties $(P1)$, $(P5)$ and $(P6)$ with the following:
	\begin{itemize}
		\item[$(P1')$] (Closure): If $n_1, n_2, n_3\in\Lambda$ are three vertices of a non-degenerate parallelogram satisfying \eqref{respq} then the fourth vertex belongs to $\Lambda$ too.
		\item[$(P5')$] (Faithfulness): Apart from $(p, q)$-families there are no other non-degenerate parallelogram satisfying \eqref{respq}.
		\item[$(P6')$] If four points $n_1, n_2, n_3, n_4$ in ${\Lambda}$ satisfy ${n}_1-{n}_2+{n}_3-{n}_4=0$ then either the relation is trivial or such points form a $(p, q)$-family.
	\end{itemize}
	
	Thanks to the above properties the only $(p, q)$-families appearing in $\Lambda$ are image of nuclear families of $\tilde{\Lambda}$.
	We call
	\[
	S_i:=\sum_{n\in \Lambda_i} |n|^{2s} \qquad i=1, \dots, N.
	\]
	Recalling that $\omega\geq 1$ and by Remark \ref{rem:monza}, we have, for $n=B \tilde{n}$,
	\[
	 q^2\, |\tilde{n}|^2 \le |n|^{2}\le p^2\, |\tilde{n}|^2\le 4 \omega^2 q^2 |\tilde{n}|^2.
	\]
	Then Theorem $4.1$ in \cite{GG21} implies the following.
	
	\begin{theorem}\label{thm:gen}
		Fix any $\tilde{\eta}>0$ small and let $s>1$. Then, there exists $\alpha>0$ large enough such that for any $N>0$ large enough and any $p,q\in\mathbb{N}$, there exists a set $\Lambda\subset p\Z\times q \Z$ with
		\[
		\Lambda:=\Lambda_1\cup\dots\cup\Lambda_N,
		\]
		which satisfies conditions $(P1'), (P2),(P3),(P4),(P5'), (P6')$ and also
		\begin{equation*}
			\frac{S_{N-2}}{S_3}\gtrsim \,2^{(s-1)(N-4)}\,\omega^{-2s}.
		\end{equation*}
		Moreover, we can ensure that each generation $\Lambda_i$ has $2^{N-1}$ disjoint frequencies satisfying 
		% (recall $R$ in Theorem \ref{thm:Iteam})
		\begin{equation}\label{bound:gen}
			\frac{S_j}{S_i}\lesssim\,e^{sN} \omega^{-2s},
		\end{equation}
		for any $1\le i< j\le N$, and 
		\begin{equation}\label{def:R}
			C^{-1}\,q\, R\le  |{n}| {\le C\,\omega\,q\,3^N\, R}, \qquad \forall {n}\in {\Lambda}_i, \qquad i=1, \dots, N,
		\end{equation}
		where $C>0$ is independent of $N$ and
		$R=R(N)$ satisfies
		% \end{equation}
	\[
	e^{\alpha^N}\leq R \leq e^{2(1+\tilde{\eta})\alpha^N}.
	\]

\end{theorem}

%We define the inner radius of the $\Lambda$ set as
%\[
%\mathtt{R}_{\Lambda}=\min_{n\in \Lambda} |n|\gtrsim q R.
%\]

%%%%%%%%%%%%%%%%%%%%%%%%%%%%%%%%%%%%%%%%%%%%%%%%%%%%%%%%%%%%%%%%%

\subsection{Normal form}\label{sec:NF}
In this section we construct a suitable set of coordinates to study the local (in the $\ell^1$-topology) dynamics of the Hamiltonian $\mathcal{H}$ in \eqref{calH}. We construct the new set of variables by using a weak version of the Birkhoff normal form procedure for PDEs around elliptic equilibriums. Essentially this consists in eliminating just a term in the Hamiltonian $\mathcal{H}$ which generates a finite dimensional vector field. For this reason we avoid the problem of small divisors, which is not present in the square (and rational) case, but it arises in the case of irrational tori.

\smallskip

We consider the finite subset $\Lambda\subset \Z^2$ given by Theorem \ref{thm:gen} and we denote by $\mathcal{H}^{(4, d)}$, $0\le d\le 4$, the Hamiltonian terms Fourier supported on 
\begin{equation}\label{def:A}
	\begin{aligned}
		\mathcal{A}(d):=\Big\{ & (n_1, n_2, n_3, n_4)\in (\Z^2)^4 : n_1-n_2+n_3-n_4=0,\\
		\,\, &\text{exactly}\,\, d\,\, \text{integer vectors $n_i$ belong to}\,\, \Z^2\setminus \Lambda \Big\}.
	\end{aligned}
\end{equation}
Thus $\mathcal{H}^{(4)}=\sum_{d=0}^4 \mathcal{H}^{(4, d)}$.
In particular $\mathcal{H}^{(4, 0)}$ is a sum of monomials $z_{n_1}\overline{ z_{n_2}} z_{n_3} \overline{z_{n_{4}}}$ where $n_1, n_2, n_3, n_{4}$ are $\Lambda$.
Since $n_1, n_2, n_3, n_4$ must satisfy the relation $n_1-n_2+n_3-n_4=0$, which is a consequence of the $x$-translation invariance of the system, $n_1, n_2, n_3, n_4$ have to form a parallelogram. By property $(P6')$ we have that $(n_1, n_2, n_3, n_4)$ must form a $(p, q)$-family or satisfy $n_1=n_2=n_3=n_4$. Therefore
\begin{equation}\label{Htoy}
\mathcal{H}^{(4, 0)}=-\frac{1}{2} \sum_{n\in \Z^2} |z_n|^4+\frac{1}{2}\sum_{\substack{(n_1, n_2, n_3, n_4)\,\mathrm{is}\\\mathrm{a\,}(p,q)\,\,\mathrm{family}}} z_{n_1} \overline{z_{n_2}} z_{n_3} \overline{z_{n_4}}.
\end{equation}
Given $n\in \Lambda_i$, $i=1, \dots, N$ we consider the equations of motion of $\mathcal{H}^{(4, 0)}$
\begin{equation}\label{eq:Htoy}
-\mathrm{i} \dot{z}_n=\partial_{\overline{z_n}} \mathcal{H}^{(4, 0)} =-|z_n|^2 z_n+2 z_{n_{\mathrm{child}\,1}} z_{n_{\mathrm{child}\,2}} \overline{z_{n\,\mathrm{spouse}}}+2 z_{n_{\mathrm{parent}\,1}} z_{n_{\mathrm{parent}\,2}} \overline{z_{n\,\mathrm{sibiling}}},
\end{equation}
where the factor $2$ comes from trivial permutations of nuclear families. The system of equations \eqref{eq:Htoy} is the one defining the toy model introduced in \cite{CKSTT}.
 Thus $\mathcal{H}^{(4, 0)}$ is the Hamiltonian of the toy model. In the normal form procedure we leave untouched this term, but we remove $\mathcal{H}^{(4, 1)}$. Indeed, if this term is not present in the Hamiltonian, the subspace
\[
\mathcal{U}_{\Lambda}=\{ z\colon \Z^2 \to \mathbb{C} : z_n=0 ,\,\,n\notin \Lambda \}
\]
is invariant by the flow of $\mathcal{H}^{(4)}$ {and the restricted Hamiltonian to $\mathcal{U}_{\Lambda}$ coincides with the Hamiltonian of the toy model, i.e. $\mathcal{H}^{(4)}_{|_{\mathcal{U}_{\Lambda}}}=\mathcal{H}^{(4, 0)}$.}

\smallskip

Given $n_1, \dots, n_4\in \Z^2$ we denote by
\begin{equation}\label{def:Omega}
	\begin{aligned}
		\Omega_{\omega}(n_1, \dots, n_4)&:=|n_1|_{\omega}^2-|n_2|_{\omega}^2+|n_3|_{\omega}^2-|n_4|_{\omega}^2,\\
		\Omega_{V}(n_1, \dots, n_4)&:=V_{n_1}-V_{n_2}+V_{n_3}-V_{n_4},\\
		\Omega_{\omega, V}(n_1, \dots, n_4)&:=\Omega_{\omega}(n_1, \dots, n_4)+\Omega_{V}(n_1, \dots, n_4).
	\end{aligned}
\end{equation}
If $\Omega_{\omega, V}(n_1, \dots, n_4)=0$ then we say that $(n_1, \dots, n_4)$ is a $4$-wave resonance for the NLS equation \eqref{NLS}.
For $k=0, \dots, 4$ let us denote 
\begin{equation}\label{def:UL}
	\mathcal{L}_k:=\inf_{\mathcal{A}(k)} |\Omega_{\omega, V}(n_1, \dots, n_4)|, \qquad \mathcal{U}_k:=\sup _{\mathcal{A}(k)} |\Omega_{\omega, V}(n_1, \dots, n_4)|.
\end{equation}
We are interested in
\begin{itemize}
\item Lower bounds for $\mathcal{L}_1$, in order to estimate the remainder given by the Birkhoff normal form procedure (see Lemma \ref{lem:L1});
\item Upper bounds for $\mathcal{U}_0$, in order to measure how far are the $(p, q)$-families to be resonant quartets (see Lemma \ref{lem:U0}). 
\end{itemize}

\medskip

We shall perform the normal form procedure in the space $\ell^1$. For $\eta>0$ we denote by 
\[
B(\eta):=\left\{  z\colon \Z^2\to \mathbb{C} : \| z\|_{\ell^1}<\eta \right\}
\]
the open ball of radius $\eta$ centered at the origin.

%
%We shall use the following classical lemma.
%\begin{lemma}\label{lem:young}
%	Let 
%	\[
%	F=\sum_{\substack{\sum_{i=1}^{d+1} \sigma_i n_i=0,\\ \sigma_i\in\{\pm\}}} F^{\sigma_1 \dots \sigma_{d+1}}_{n_1 \dots n_{d+1}}\,z^{\sigma_1}_{n_1}\dots z^{\sigma_{d+1}}_{n_{d+1}}, \qquad z_n^{+}:=z_n, \quad z_n^{-}:=\overline{z_n}
%	\]
%	be a homogenous Hamiltonian of degree $d+1$ preserving momentum. Then
%	\[
%	\| X_F(z) \|_{\ell^1}\lesssim \bral F \brar \| z\|_{\ell^1}^d,
%	\]
%	where
%	\[
%	\bral F \brar:=\sup_{(\sigma_i, n_i)} |F^{\sigma_1 \dots \sigma_{d+1}}_{n_1 \dots n_{d+1}}|.
%	\]
%	Moreover, let $G$ be a homogenous, momentum preserving, Hamiltonian of degree $d'+1$ of the same form of $F$. Then $\{ F, G \}$ is a homogenous, momentum preserving, Hamiltonian of degree $d+d'$ and
%	\[
%	\bral \{F, G\} \brar\lesssim \bral F \brar\,\bral G \brar.
%	\]
%	{
%		We deduce that
%		\begin{equation}\label{ineq:lie}
%			\| [X_F, X_G](z) \|_{\ell_1}\lesssim \bral F \brar\,\bral G \brar \| z\|^{d+d'-1}_{\ell^1},
%		\end{equation}
%		where $[\cdot, \cdot ]$ denotes the standard Lie bracket between vector fields.
%	}
%	
%\end{lemma}

\begin{lemma}\label{lem:L1}
	Recall the constant $R$ of Theorem \ref{thm:gen}  and the $\psi$-convergent $(p, q)$ of $\omega$ in \eqref{B}. There exists a universal constant $c>0$ such that if
	 \begin{equation}\label{assumption}
		3^{2N}\,R^2\, \frac{\psi(q)}{q}\le c
	\end{equation}
	then
	\begin{equation}\label{lower_bound}
		\mathcal{L}_1\geq \mathtt{L},
	\end{equation}
where $\mathtt{L}$ is introduced in \eqref{cond:M}.
\end{lemma}

\begin{proof}
	By Lemma $4.5$ in \cite{GG21}, which requires the assumption \eqref{assumption}, we have that
	\[
	\inf_{\mathcal{A}(1)} |\Omega_{\omega}(n_1, \dots, n_4)|\geq  \frac{p^2}{2}\geq \frac{\omega^2 q^2}{8} .
	\]
	
	By triangle inequality we have
	\[
	|\Omega_V(n_1, \dots, n_4)|\le 4 \sup_n |V_n|.
	\]
	Thus, recalling \eqref{cond:M},
	\[
	|\Omega_{\omega, V}(n_1, n_2, n_3, n_4)|\geq |\Omega_{\omega}(n_1, \dots, n_4)|-|\Omega_V(n_1, \dots, n_4)| \geq \frac{\omega^2}{8} q^2-4 \sup_n |V_n|=\mathtt{L}.
	\]

\end{proof}

{
\begin{remark}{(Case $\omega=1$)}
In the square case one can prove that $\mathcal{L}_1\gtrsim q^2\,R^2$. Indeed $\Omega_1(n_1,\dots, n_4)\neq 0$, because of the closure property $(P1)$, and $n_i$ are integer vectors with $|n_i|\gtrsim q R$.
\end{remark}
}

%\begin{remark}
%In \cite{GG21} the lower bound \eqref{lower_bound} is better, because of the absence of the potential $V$. Anyway, we do not need this better estimate to conclude the argument, then this allows to weaken the conditions \eqref{cond:M}. \red{As a consequence we do not worse the upper bound on the time of instability $T$. (?)}
%\end{remark}

We observe that the quantity $\mathtt{L}$ in \eqref{cond:M} is increasing with respect to the parameter $q$. Hence we can obtain large lower bounds on $\mathcal{L}_1$ choosing large $q$'s. This fact shall be used to make smaller the remainder of the following Birkhoff normal form procedure.

\begin{proposition}[Weak Birkhoff Normal Form]\label{prop:wbnf}
	There exist $\eta_0>0$ and a symplectic change of coordinates $z=\Gamma(w)$ such that for all $\eta\in (0, \eta_0)$ the map $\Gamma\colon B(\eta)\to B(2\eta)$ transforms the Hamiltonian $\mathcal{H}$ in \eqref{calH} into the Hamiltonian
	\begin{equation}\label{def:Htilde}
	\mathcal{H}\circ \Gamma=\mathcal{H}^{(2)}+\mathcal{H}^{(4, 0)}+\mathcal{H}^{(4, \geq 2)}+\mathcal{R},
	\end{equation}
	where 
	\begin{equation}\label{bound:remainder}
	\sup_{w\in B(\eta)} \| X_{\mathcal{R}}(w)\|_{\ell^1} \lesssim \mathtt{L}^{-1} \eta^5.
	\end{equation}
	Moreover
 the maps $\Gamma^{\pm 1}$ are (locally) close to the identity in the following sense
	\begin{equation}\label{bound:GammaId} 
	\sup_{w\in B(\eta)}\| \Gamma^{\pm 1}(w)- w \|_{\ell^1} \lesssim \,\mathtt{L}^{-1}\, \eta^3.
	\end{equation}
	
	%if $\eta$ is such that
	%\begin{equation}\label{cond:remainder}
	%\eta^2 \lesssim \mathcal{L}_1
	%\end{equation}

\end{proposition}

\begin{proof}
	The proof follows by Proposition $1$ in \cite{GG21} and by Lemma \ref{lem:L1}. 
	
%	
%	The only thing which is not proven in \cite{GG21} is the bound \eqref{radius} on the radius $\eta$ of the domain where the Birkhoff map $\Gamma$ is well defined. This map is the time-one flow map of a $4$-degree homogenous Hamiltonian
%	\[
%	F=\sum_{\mathcal{A}(1)} F_{n_1 \dots n_4}^{\sigma_1 \dots \sigma_4} z_{n_1}^{\sigma_1} \dots z_{n_4}^{\sigma_4}
%	\] which solves the following homological equation
%	\begin{equation}\label{homo}
%	\{ F, \mathcal{H}^{(2)} \}+\mathcal{H}^{(4)}=\mathcal{H}^{(4, 0)}+\mathcal{H}^{(4, \geq 2)}.
%	\end{equation}
%	Observe that $\mathcal{A}(1)$ is a finite set, hence $F$ defines a finite dimensional ODE system.
%	 To ensure the existence of the flow $\Phi^t_F$ for times $t\in [0, 1]$ and that $\Phi^t_F\colon B(\eta)\to B(2\eta)$ one can use a standard fixed point argument, which requires that
%	\[
%	 \eta\,\left(\sup_{n_i, \sigma_i} |F_{n_1 \dots n_4}^{\sigma_1 \dots \sigma_4}|\right)^{1/2}
%	 \] is small enough.  By solving the homological equation \eqref{homo} in Fourier and by using Lemma \ref{lem:L1} we have that
%	 \[
%	 |F_{n_1 \dots n_4}^{\sigma_1 \dots \sigma_4}|\lesssim \mathtt{L}^{-1}.
%	 \]
%	Hence $\Gamma$ is well defined on $B(\eta)$ with $\eta\lesssim \mathtt{L}^{1/2}$. In order to get the estimate for the remainder \eqref{bound:remainder} we need to ask that $\eta^5 \mathtt{L}^{-1}$ is small enough.
	\end{proof}

\subsection{Quasi-resonant model}\label{sec:resmodel}

In this section we derive the toy model of \cite{CKSTT} as a quasi-resonant system of the Hamiltonian in Birkhoff coordinates \eqref{def:Htilde}. Then we apply a theorem proved in \cite{GuardiaK12} which ensures the existence of an orbit of the toy model which displays the desired energy exchange behavior and give an upper bound on the time at which the energy is transferred to high modes. 
\begin{remark}
By Theorem \ref{thm:gen} low and high modes are respectively contained in generations $\Lambda_3$ and $\Lambda_{N-2}$. The generations $\Lambda_1, \Lambda_2, \Lambda_{N-1}$ and $\Lambda_N$ are not considered just for technical reasons. We refer to \cite{CKSTT} for more details.
\end{remark}

We introduce the rotating coordinates
\begin{equation}\label{def:rotcoord}
	w_{n}=r_{n}\,e^{\mathrm{i} (|n|_{\omega}^2+V_n)\,t}, \qquad n\in \Z^2.
\end{equation}
Even if this change of variables is not canonically symplectic, the new vector field $\mathcal{X}$ still possesses a Hamiltonian structure, with a time-dependent Hamiltonian, and it is given by

\begin{equation*}
	\mathcal{X}:= X_{\mathtt{H}_0}+ X_{\mathtt{H}_{\geq 2}}+X_{\mathtt{R}},
\end{equation*}
where (recall notation \eqref{def:A})
\begin{align*}
	&\mathtt{H}_0:=\mathcal{H}^{(4, 0)}(\{ r_{n}\,e^{\mathrm{i} \lambda(n)t} \}_{n\in\Lambda}),\\
	&\mathtt{H}_{\geq 2}:=\mathcal{H}^{(4, \geq 2)}(\{ r_{n}\,e^{\mathrm{i} \lambda(n)t} \}_{n\in\Z^2}),\\
	&\mathtt{R}(t)=\mathcal{R} (\{ r_{n}\,e^{\mathrm{i} \lambda(n)t} \}_{n\in\Z^2}).
\end{align*}
Defining
\begin{equation}\label{def:N}
	\mathcal{N}:=\mathcal{H}^{(4, 0)}+\mathcal{H}^{(4, \geq 2)}\qquad \text{ and }\qquad {\mathcal{Q}^{(4, 0)}}:=\mathtt{H}_0-\mathcal{H}^{(4, 0)}, \qquad {\mathcal{Q}^{(4, \geq 2)}}:=\mathtt{H}_{\geq 2}-\mathcal{H}^{(4, \geq 2)}
\end{equation}
we can write $\mathcal{X}$ as 
\begin{equation}\label{vecfrot}
	\mathcal{X}=X_{\mathcal{N}}+X_{\mathcal{Q}^{(4, 0)}}+X_{\mathcal{Q}^{(4, \geq 2)}}+X_{\mathtt{R}}.
\end{equation}
The Hamiltonian $\mathcal{N}$ gives the first order of the Hamiltonian in rotating coordinates. Its flow leaves invariant the subspace
\[
\mathcal{V}_{\Lambda}:=\{ r\colon \Z^2\to\C : r_{n}=0, \,\,\, n\notin \Lambda  \},
\]
since the vector field of $\mathcal{H}^{(4, \geq 2)}$ vanishes on it.
 The dynamics of $\mathcal{N}$ on $\mathcal{V}_{\Lambda}$ is given by the restricted Hamiltonian
 \[
 \mathcal{N}_{|_{\mathcal{V}_{\Lambda}}}=\mathcal{H}^{(4, 0)}.
 \]
 As we commented at the beginning of Section \ref{sec:NF}, the Hamiltonian $\mathcal{H}^{(4, 0)}$ is the Hamiltonian of the toy model introduced in \cite{CKSTT}.
 By the fact that the frequencies from each generation interact with that generation and with its adjacent generations in exactly the same way we have the following.
 
%\begin{remark}\label{rem}
%	Observe that the vector fields $X_{\mathcal{H}^{(4, \geq 2)}}$ and $X_{\mathtt{H}^{(4, \geq 2)}}$ vanish on $\mathcal{V}_{\Lambda}$. 
%\end{remark}

\begin{lemma}[Intragenerational equality \cite{CKSTT}]
	Consider the subspace
	\[
	\widetilde{\mathcal{V}_{\Lambda}}:=\left\{ r\in \mathcal{V}_{\Lambda} : r_{n}=r_{n'}\,\,\,\forall n, n'\in \Lambda_j\,\,\text{for some}\,\,j\in \{ 1, \dots, N \} \right\}
	\]
	where all the members of a generation take the same value. Then $\widetilde{\mathcal{V}_{\Lambda}}$ is invariant under the flow of \eqref{Htoy}.
\end{lemma}
We set
\begin{equation*}
	b_i=r_n \qquad \text{for any}\qquad n\in \Lambda_i, \quad i=1, \dots, N.
\end{equation*}
The equation \eqref{eq:Htoy} restricted on $\widetilde{\mathcal{V}_{\Lambda}}$ reduces to the following nearest-neighbor system
\begin{equation}\label{eq:toy}
	\dot{b}_i=-\mathrm{i} b_i^2 \overline{b_i}+2\mathrm{i} \overline{b_i} (b^2_{i-1}+b^2_{i+1}), \qquad i=1, \dots, N.
\end{equation}
The following result concerns the existence of an orbit of \eqref{eq:toy} which displays a forward energy cascade behavior.
\begin{theorem}{(\cite{GuardiaK12})}\label{thm:orbit}
	Fix $\gamma>0$ large enough. Then for any large enough $N$ and $\delta=\exp(-\gamma N)$, there exists a trajectory $b(t)$ of the system \eqref{eq:toy}, $\sigma>0$ independent of $\gamma, N$ and $T_0>0$ such that
	\begin{align*}
		|b_3(0)|>1-\delta^{\sigma},& \qquad |b_j(0)|<\delta^{\sigma} \quad\,\,\, \mathrm{for}\,\,j\neq 3,\\
		|b_{N-1}(T_0)|>1-\delta^{\sigma},& \qquad |b_j(T_0)|<\delta^{\sigma} \quad  \mathrm{for}\,\,j\neq N-1.
	\end{align*}
	Moreover there exists a constant $\mathbb{K}>0$ independent of $N$ such that $T_0$ satisfies
	\begin{equation}\label{bound:timeT0}
		% \mathbb{K}^{-1} \gamma N^2\leq 
		0<T_0\leq  \mathbb{K}\,N\,\log(\delta^{-1})=\mathbb{K} \gamma N^2.
	\end{equation}
\end{theorem}
% Note that \cite{GuardiaK12} only provides an upper bound for the time. However, from the proof in  \cite{GuardiaK12}, one can easily obtain lower bounds of the same order. 

Since $\mathcal{N}$ in \eqref{def:N} is a homogenous Hamiltonian of degree $4$ we can consider the scaled solution
\begin{equation*}\label{lambda}
	b^{\lambda}(t)=\lambda^{-1} b(\lambda^{-2} t), \qquad \lambda>1,
\end{equation*}
where $b(t)$ is the trajectory given by Theorem \ref{thm:orbit}.
%
%
%Let us call \blue{(Is this correct? Maybe it could be written better)}
%\[
%\mathtt{r}(t):=\sum_{i=1}^N\sum_{n=(j, k)\in\Lambda_i} \mathtt{r}_n(t)\,e^{\mathrm{i} (j x+\omega k y)}
%\]
%where all the $\mathtt{r}_n(t)$ with $n\in \Lambda_i$ coincides with the orbit $b_i(t)$ given by Theorem \ref{thm:orbit}, which has life-span $[0, T_0]$ (recall the bound \eqref{bound:timeT0}).  The function $\mathtt{r}(t)$ is solution of the Hamiltonian system $\mathcal{N}$.\\
Then
\begin{equation}\label{def:rlambda}
	r^{\lambda}(t, x)=\sum_{n\in \Z^2} r^{\lambda}_n(t)\,e^{\mathrm{i} n \cdot x}, \qquad r^{\lambda}_n(t):=\begin{cases}
		b_i^{\lambda}(t) \qquad n\in\Lambda_i, \quad i=1, \dots, N,\\
		0 \qquad \,\, \,\, \quad n\notin \Lambda
	\end{cases}
\end{equation}
is a solution of $\mathcal{N}$ in \eqref{def:N}.
The life-span of $r^{\lambda}(t)$ is $[0, T]$ where
\begin{equation}\label{def:timeT}
	T:=\lambda^{2} T_0=\lambda^2 \gamma \mathbb{K} N^2.
\end{equation}
By Theorem \ref{thm:orbit} and arguing as in Lemma $9.2$ in \cite{GuardiaHP16}, we have
\begin{equation}\label{bound:base}
	\sup_{t\in [0, T]} \| r^{\lambda}(t) \|_{\ell^1}\lesssim N 2^{N-1} \lambda^{-1}.
\end{equation}

\subsection{Approximation argument}\label{sec:approxarg}
In this section we prove that there exists a solution of \eqref{vecfrot} that stays close to $r^{\lambda}(t)$ in the $\ell^1$-topology over the time interval $[0, T]$.
%To this end, we give an upper bound on the resonant combinations among modes in $\Lambda$. This will be useful in estimating the error generated by considering a quasi-resonant model instead of a resonant one.
We set
\begin{equation}\label{delta}
	\vartheta:=3^{2N} R^2\,\omega^3\,q \psi(q)+ 4 (q R)^{-s_0}.
\end{equation}
Recalling the property \eqref{cond:psi} on the function $\psi$ we see that $\vartheta$ can be made arbitrarily small by taking $q$ large enough.

\smallskip

We need some preliminary lemmata. The first lemma gives an upper bound on $\mathcal{U}_0$, introduced in \eqref{def:UL}, which measures how close is a $(p, q)$-family to be a $4$-wave resonance. The second one gives estimates on the vector field of homogeneous Hamiltonians as $\mathcal{N}, \mathcal{Q}^{(4, 0)}, \mathcal{Q}^{(4, \geq 2)}$ appearing in \eqref{vecfrot}. Since the latter is a classical result we omit the proof, which is based on Young's inequality.

\begin{lemma}\label{lem:U0}
	%Recall the definition \eqref{def:UL}.
	%Assume that $\omega$ is $\psi$-approximable.
	 The constant $\mathcal{U}_0$ in \eqref{def:UL} satisfies
	\[
	\mathcal{U}_0\lesssim  \vartheta.
	\]
\end{lemma}
\begin{proof}
By property $(P6')$ if $(n_1, n_2, n_3, n_4)\in \mathcal{A}(0)$ (recall the definition of $\mathcal{A}(0)$ in \eqref{def:A}) then $(n_1, n_2, n_3, n_4)$ forms a $(p, q)$-family or $n_1=n_2=n_3=n_4$.
	Then by Lemma $6.1$ in \cite{GG21} we have that
	\[
	\left|\Omega_{\omega}(n_1, n_2, n_3, n_4)\right|\lesssim \max_{n\in\Lambda}\{ |n|^2 \}\,\omega\,\frac{\psi(q)}{q}\stackrel{\eqref{def:R}}{\lesssim} 3^{2N} R^2\,\omega^3\, q\, {\psi(q)} \qquad \forall (n_1, n_2, n_3, n_4)\in \mathcal{A}(0).
	\]
	By \eqref{def:R} $|n_i|\gtrsim q R $, hence by the decay of the Fourier coefficients of $V$ \eqref{decaypot} we have
	\[
	|\Omega_V(n_1, n_2, n_3, n_4)|\lesssim 4 (q R)^{-s_0} \qquad \forall (n_1, n_2, n_3, n_4)\in \mathcal{A}(0).
	\]
	By triangle inequality we get the thesis.
\end{proof}

{
\begin{remark}{(Case $\omega=1$)}
In the square case one can just consider $\vartheta$ in \eqref{delta} as $\vartheta=4 (q R)^{-s_0}$.
\end{remark}
}

\begin{lemma}\label{lem:young}
Let 
\[
F=\sum_{\substack{\sum_{i=1}^{d+1} \sigma_i n_i=0,\\ \sigma_i\in\{\pm\}}} F^{\sigma_1 \dots \sigma_{d+1}}_{n_1 \dots n_{d+1}}\,r^{\sigma_1}_{n_1}\dots r^{\sigma_{d+1}}_{n_{d+1}}, \qquad r_n^{+}:=r_n, \quad r_n^{-}:=\overline{r_n}
\]
 be a homogenous Hamiltonian of degree $d+1$ preserving momentum. Then
\[
\| X_F(r) \|_{\ell^1}\lesssim \bral F \brar \| r\|_{\ell^1}^d \qquad \forall r\in \ell^1,
\]
where
\[
\bral F \brar:=\sup_{(\sigma_i, n_i)} |F^{\sigma_1 \dots \sigma_{d+1}}_{n_1 \dots n_{d+1}}|.
\]
Moreover, for all $r, r'\in \ell^1$,
 \begin{align*}
 \| X_F(r)-X_F(r') \|_{\ell^1} &\lesssim \bral F \brar (\| r \|_{\ell^1}+\| r' \|_{\ell^1})^{d-1} \,\| r-r' \|_{\ell^1},\\
 \| X_F(r)-X_F(r')-D X_F(r) [r-r'] \|_{\ell^1} &\lesssim  \bral F \brar \| r \|_{\ell^1}^{d-2} \| r-r' \|^2_{\ell^1},
 \end{align*}
 where $D X_F(r)$ denotes the differential of the vector field at $r$.
\end{lemma}

\begin{proposition}{(Approximation argument)}\label{prop:approxarg}
	%Recall $R$ in \eqref{def:R}, $\gamma$ in Theorem \ref{thm:orbit} and \eqref{set:qlambda}.\\ 
	Fix $\epsilon>0$ small enough. Then, for $q$ satisfying the assumption \eqref{assumption} of Lemma \ref{lem:L1},
	\begin{equation}\label{cond:final}
		\lambda={2}^{s N},\qquad \vartheta\le \exp(-5^N), \qquad \mathtt{L}\geq \exp(5^N)
	\end{equation}
	and for $N>0$ large enough
	the following holds. If $r(t)$ is a solution of \eqref{vecfrot} such that
	\begin{equation}\label{assump:initial}
		\| r(0)-r^{\lambda}(0) \|_{\ell^1}\le \lambda^{-2}\mathtt{L}^{-1/2}
	\end{equation}
	then
	\begin{equation}\label{bound:close}
		\sup_{t\in [0, T]}\| r(t)-r^{\lambda}(t) \|_{\ell^1}\le \lambda^{-(1+\epsilon)},
	\end{equation}
	where $T$ is the time in \eqref{def:timeT}.
\end{proposition}

\begin{proof}
By the above assumptions we can apply Lemma \ref{lem:L1}  and consider the estimates \eqref{bound:remainder}, \eqref{bound:GammaId}.

	We define $\xi(t):=r(t)-r^{\lambda}(t)$. The equation for $\xi$ can be written as $\dot{\xi}=Z_0+Z_1+Z_2+Z_3+Z_4$,
	where
	\begin{align*}
		&Z_0:=X_{\mathtt{R}}(r^{\lambda}+\xi),\\
		&Z_1:=D X_{\mathcal{N}}(r^{\lambda})[\xi],\\
		&Z_2:=X_{\mathcal{N}}(r^{\lambda}+\xi)-X_{\mathcal{N}}(r^{\lambda})-D 
		X_{\mathcal{N}}(r^{\lambda})[\xi],\\
		&Z_3:=X_{\mathcal{Q}^{(4, 0)}}(r^{\lambda})+X_{\mathcal{Q}^{(4, \geq 2)}}(r^{\lambda}),\\
		&Z_4:=X_{\mathcal{Q}^{(4, 0)}}(r^{\lambda}+\xi)-X_{\mathcal{Q}^{(4, 0)}}(r^{\lambda})+X_{\mathcal{Q}^{(4, \geq 2)}}(r^{\lambda}+\xi)-X_{\mathcal{Q}^{(4, \geq 2)}}(r^{\lambda}).
	\end{align*}
	By the differential form of Minkowski's inequality we obtain
	\[
	\frac{d}{d t} \| \xi \|_{\ell^1}\le \|Z_0\|_{\ell^1}+\|Z_1\|_{\ell^1}+\|Z_2\|_{\ell^1}+\|Z_3\|_{\ell^1}+\|Z_4\|_{\ell^1}.
	\]
	We assume temporarily a \emph{bootstrap assumption}. We call $T^*$ to the 
	supremum of the times $T'>0$ such that
	\[
	\| \xi(t)\|_{\ell^1}\le 2   \lambda^{-(1+\epsilon)} \qquad \forall t\in [0, 
	T'].
	\]
	Note that \eqref{assump:initial} implies $T_*>0$. A posteriori we will prove 
	that  $T_*\geq T$, where $T$ is the time introduced in \eqref{def:timeT}, and therefore drop the bootstrap assumption.
	
	First we need a priori estimates on the terms $Z_i$, $i=1, 2, 3, 4$, defined above.  {We shall 
	use repeatedly Lemma \ref{lem:young}} and the bootstrap assumption without 
	mentioning them. We remark that the change to rotating coordinates 
	\eqref{def:rotcoord} does not affect the bounds on the $\ell^1$ norm.
	
	\medskip
	
	\noindent\textbf{Bound for $Z_0$.} By \eqref{bound:base} and the bootstrap assumption
	\[
	\left\|r^\lambda+\xi\right\|_{\ell^1}\lesssim N 2^{N-1} 
	\lambda^{-1}+\lambda^{-(1+\epsilon)}\lesssim N 2^N \lambda^{-1}.
	\]
	Hence by \eqref{bound:remainder}
	\[
	\| Z_0 \|_{\ell^1}\lesssim \mathtt{L}^{-1}  N^5 2^{5N} 
	\lambda^{-5}.
	\]
	\noindent\textbf{Bound for $Z_1$.} By \eqref{bound:base} and considering that $\mathcal{N}$ is a homogenous Hamiltonian of degree $4$ we have 
	\[
	\| Z_1 \|_{\ell^1}\lesssim N^2 2^{2(N-1)} \lambda^{-2} \| \xi \|_{\ell^1}\lesssim 
	N^2 4^{N} \lambda^{-2}\, \| \xi \|_{\ell^1}.
	\]
	
	\noindent\textbf{Bound for $Z_2$.} By 
	% \eqref{bound:remainder2} and 
	\eqref{bound:base} and using the bootstrap assumption we have
	\[
	\| Z_2 \|_{\ell^1}\lesssim 
	% q^{-2}\, N 2^{N-1}^4 \lambda^{-4} \| \xi \|_{\ell^1}+q^{-4}\,N 2^{N-1}^6 
	% \lambda^{-6} \| \xi \|_{\ell^1} +
	\lambda^{-1} N 2^{N-1} \| \xi\|_{\ell^1}^2\lesssim 
	 N 2\,^{N}\, \lambda^{-2-\epsilon}\, \| \xi \|_{\ell^1}.
	\]
	% By using the bootstrap assumption and the condition \eqref{cond:final} 
	% \[
	%  \| Z_2 \|_{\ell^1}\lesssim \Big( \, N 2^{N-1}^3 \lambda^{-2\nu-2}+  N 2^{N-1}^5 
	% \lambda^{-4\nu-4} + \lambda^{-\epsilon} \Big)\,\lambda^{-2}\,N 2^{N-1}\, \| 
	% \xi\|_{\ell^1}\lesssim N 2\,^{N-1}\, \lambda^{-2}\, \| \xi \|_{\ell^1}
	% \]
	%\nu<1
	% provided that $\alpha$ is large enough.\\
	\noindent\textbf{Bound for $Z_3$.} We have $X_{\mathcal{Q}^{(4, 
			\geq 2)}}(r^{\lambda})=0$ because the vector field of $\mathcal{Q}^{(4, 
			\geq 2)}$ vanishes on $\mathcal{V}_{\Lambda}$. Then, by Lemma \ref{lem:U0}, we have
	\[
	\| Z_3(t) \|_{\ell^1}\lesssim \mathcal{U}_0 N^3 2^{3N}\,  \lambda^{-3}\,t\lesssim 
	\vartheta N^3 2^{3N}\,  \lambda^{-3}\,t.
	\]
	where we have used the bound
%	% Lemma \ref{lem:U0}
%	% and the bound
	\[
	|e^{\mathrm{i} a t}-1|\le |a t| \qquad \forall a, t \in \R.
	\]
	\noindent\textbf{Bound for $Z_4$.} By \eqref{bound:base} we have
	\[
	\| Z_4 \|_{\ell^1}\lesssim   
	N^2 4^{N} \lambda^{-2}\, \| \xi \|_{\ell^1},
	\]
	where we have used that $|e^{\mathrm{i}\alpha}-1|\le 2$ for all $\alpha\in \R$.
	
	\medskip
	
	By collecting the previous estimates we have
	\[
	\frac{d}{dt} \| {\xi}(t) \|_{\ell^1}\le C_0( N^2 4^{N} \lambda^{-2} \| \xi(t) 
	\|_{\ell^1} + N^3 2^{3N}\, \lambda^{-3}\,\vartheta\,t+\mathtt{L}^{-1} N^5 
	2^{5N} \lambda^{-5}),
	\]
	for some universal constant $C_0>0$. Then, by Gronwall Lemma and condition  
	\eqref{cond:final},
	\begin{align*}
		\| \xi(t) \|_{\ell^1}
		&\le  \left(\lambda^{-2} \mathtt{L}^{-1/2}+ 
		N^3 2^{3 N-1} \vartheta \lambda^{-3}  t^2
		+\mathtt{L}^{-1} N^5 2^{5N} \lambda^{-5}\,  t  \right)\,\exp 
		\left(C_0\, N^2 4^{N} \lambda^{-2} t \right).
	\end{align*}
	
	Then, for $t\in [0, T]$, where $T$ is the time defined in \eqref{def:timeT}, one has  
	\begin{align*}
		\| \xi(t) \|_{\ell^1}
		&\le  \left(\lambda^{-2} \mathtt{L}^{-1/2}+{ N^7 2^{3 N-1} \vartheta \lambda\, \mathbb{K}^2 \gamma^2}+\mathtt{L}^{-1}\,{N^7 2^{5N} \lambda^{-3} \mathbb{K} \gamma}\,\right)\,\exp(\mathbb{K} \gamma N^4 4^N C_0).
	\end{align*}
	We impose the following conditions
	\begin{align}
	 \frac{3}{2}\,\exp(\mathbb{K} \gamma N^4 4^N C_0)< \,\lambda^{1-\epsilon} \mathtt{L}^{1/2}, \label{miss1}\\
	\vartheta\, <  \frac{2}{3} \lambda^{-2-\epsilon}\, N^{-7} 2^{-3N+1} \gamma^{-2} \mathbb{K}^{-2} \exp(-\mathbb{K} \gamma N^4 4^N C_0), \label{miss2} \\
	\,\mathbb{K} \gamma N^7 2^{5N}\,\left(\frac{3}{2} \exp(\mathbb{K} \gamma N^4 4^N C_0) \right)  <  \,\lambda^{2-\epsilon} \mathtt{L}. \label{miss3}
	\end{align}
	Thanks to the assumptions \eqref{cond:final} the above conditions are satisfied, provided that $N>0$ is taken large enough.
%	For $N$ large enough condition \eqref{miss1} implies \eqref{miss3} thanks to the first assumption in \eqref{cond:final}.
%	{We consider $N>0$ large enough such that
%	\[
%	 \exp(-\mathbb{K} \gamma N^4 4^N C_0)\geq  \lambda^{-\epsilon}.
%	\]
%	}
%	Then the second assumption in \eqref{cond:final} implies \eqref{miss2} for $N$ large enough.
	Hence we have that
	\[
	\| \xi( t) \|_{\ell^1}< 2  \lambda^{-(1+\epsilon)} \qquad \forall t\in [0, T].
	\]
	Therefore, $T_*\geq T$ and  we can drop the bootstrap assumption. This concludes the proof.
\end{proof}

\subsection{Growth of Sobolev norms for NLS solutions}\label{sec:conclusion}

The aim of this section is to prove the following result.

\begin{lemma}\label{lem:ratio}
	Fix $s>1$. Then, for $N>0$ large enough and assuming the assumptions of Proposition \ref{prop:approxarg} the following holds.
	There exists a solution $z(t)$ of \eqref{calH} such that
	\begin{equation}\label{ratio}
		\frac{\| z(T) \|_s^2}{\| z(0) \|_s^2}\gtrsim 2^{(s-1) (N-6)}\,\omega^{-2s},
	\end{equation}
	where $T$ is the time introduced in \eqref{def:timeT}.
\end{lemma}

\begin{proof}

	We consider the solution $z(t)$ of the Hamiltonian system \eqref{calH} such that 
	\begin{equation}\label{agri}
	z(0)=r^{\lambda}(0),
	\end{equation} where $r^{\lambda}(t)$, defined in \eqref{def:rlambda}, is the solution of the truncated Hamiltonian $\mathcal{N}$ in \eqref{def:N} exhibiting the energy cascade behavior. 
	We note that, by \eqref{bound:base},
	\begin{equation}\label{giorgi}
	\| z(0)\|_{\ell^1}=\| r^{\lambda}(0) \|_{\ell^1}\lesssim  N 2^{N-1} \lambda^{-1}.
	\end{equation}
	By condition \eqref{cond:final} on $\lambda$ and provided that $N$ is large enough, $z(0)$ belongs to the ball of $\ell^1$ where the Birkhoff map (given in Proposition \ref{prop:wbnf}) is well defined.
	
	Using the approximation argument (Proposition \ref{prop:approxarg}) we want to show that $z(t)$ inherits from $r^{\lambda}(t)$ the property of having most of the energy concentrated initially on low modes and that, at later times, some of the energy has been passed to high modes. To compare $z(t)$ and $r^{\lambda}(t)$ we need to write $z(t)$ in Birkhoff and rotating coordinates, respectively given in Proposition \ref{prop:wbnf} and \eqref{def:rotcoord}.
	
	We write the solution $z(t)$ as
	\begin{equation}\label{giorgi2}
	z(t)=\Gamma\left(\{ r_n(t)\,e^{\mathrm{i} \lambda(n) t} \}_n \right),
	\end{equation}
	where $r(t)=\sum_{n\in \Z^2} r_n(t)\,e^{\mathrm{i} n x}$ is solution of the system \eqref{vecfrot} (the normalized equation in rotating coordinates). 
	\begin{remark}\label{otto}
	We shall verify that $r(t)$ stays in the $\ell^1$-domain of definition of the Birkhoff map $\Gamma$ over the range of time $[0, T]$.
	\end{remark}
	
	Since the Birkhoff map (and its inverse) are locally close to the identity (recall the bound on $\mathtt{L}$ in \eqref{cond:final}) and $z(0)$ satisfies \eqref{giorgi} we have that
		\begin{equation}\label{mildred}
		\| r(0) \|_{\ell^1}=\|  \Gamma^{-1} (z(0))\|_{\ell^1}\le \| z(0) \|_{\ell^1}+\| \Gamma^{-1} (z(0))-z(0)\|_{\ell^1}\lesssim \lambda^{-1} N\,2^N.
	\end{equation}
	Now we show that $r(t)$ fits into the assumption of Proposition \ref{prop:approxarg} and so it is well approximated by the trajectory $r^{\lambda}(t)$ given in \eqref{def:rlambda}.
	To this end, we need to check condition \eqref{assump:initial}.
	By the bound \eqref{mildred}, the closeness to the identity of the Birkhoff map, the conditions on the initial data \eqref{agri}, on $\lambda$ \eqref{cond:final} and on $\mathtt{L}$ in \eqref{cond:final} we have, provided that $N$ is large enough,
	\begin{equation}\label{bound:diff}
	\begin{aligned}
		\| r(0)-r^{\lambda}(0) \|_{\ell^1} &\le \| r(0)-z(0) \|_{\ell^1}+\| z(0)-r^{\lambda}(0) \|_{\ell^1}\le \| r(0)-\Gamma(r(0)) \|_{\ell^1}\\ 
		&\le C \mathtt{L}^{-1} N^3 2^{3(N-1)} \lambda^{-3}\le \lambda^{-2} \mathtt{L}^{-1/2},
		%\stackrel{\eqref{bound:GammaId2}}{\lesssim} {q^{-2}} \| r(0) \|_{\ell^1}^3.
		\end{aligned}
		\end{equation}
		where $C>0$ is a universal constant.

	Hence, the assumption \eqref{assump:initial} is satisfied and we can apply the approximation argument (Proposition \ref{prop:approxarg}) for a fixed $\epsilon>0$. Therefore
	\begin{equation}\label{mildred2}
	\sup_{t\in [0, T]}\| r(t)-r^{\lambda}(t) \|_{\ell^1}\le \lambda^{-(1+\epsilon)}.
	\end{equation}
	By \eqref{bound:base} and \eqref{mildred2}, we have, for all $t\in [0, T]$,
	\begin{equation}\label{agri2}
	 \| r(t) \|_{\ell^1}\le \| r^{\lambda}(t) \|_{\ell^1}+\| r(t)-r^{\lambda}(t) \|_{\ell^1}\le C  N 2^{N-1} \lambda^{-1}-\lambda^{-(1+\epsilon)}\lesssim N 2^{N-1} \lambda^{-1}.
	\end{equation}
	Hence, since $\lambda$ satisfies \eqref{cond:final} and $N$ is large enough, \eqref{giorgi2} is well defined (recall Remark \ref{otto}).
	Now we use Proposition \ref{prop:approxarg} to show that 
	\begin{equation}\label{imperiale}
	|z_n(T)|\geq \frac{\lambda^{-1}}{2} \qquad \forall n\in \Lambda_{N-2}.
	\end{equation} 
	We have that
	\begin{equation}
	\begin{aligned}
	|z_n(T)|&\geq |r_n(T)|-|\Gamma_n(\{ r_n(t)\,e^{\mathrm{i} \lambda(n) t} \}_n)- r_n(t)\,e^{\mathrm{i} \lambda(n) t} |\\
	&\geq |r_n^{\lambda}(T)|-|r_n(T)-r_n^{\lambda}(T)|-|\Gamma_n(\{ r_n(t)\,e^{\mathrm{i} \lambda(n) t} \}_n)- r_n(t)\,e^{\mathrm{i} \lambda(n) t} |.
	\end{aligned}
	\end{equation}
By the construction performed in Section \ref{sec:resmodel}, in particular by Theorem \ref{thm:orbit}, and the definition of $r^{\lambda}(t)$ in \eqref{def:rlambda} we have
\[
|r_n^{\lambda}(T)|=\lambda^{-1} |b_{N-2}(T_0)|\geq \frac{3}{4} \lambda^{-1} \qquad \forall n\in \Lambda_{N-2}.
\]
By \eqref{mildred2} 
\[
|r_n(T)-r_n^{\lambda}(T)|\le \lambda^{-(1+\epsilon)}.
\] 
By \eqref{agri2} and \eqref{bound:GammaId} we have
\[
|\Gamma_n(\{ r_n(T)\,e^{\mathrm{i} \lambda(n) T} \}_n)- r_n(T)\,e^{\mathrm{i} \lambda(n) T} |\lesssim \mathtt{L}^{-1} N^3 2^{3(N-1)} \lambda^{-3}.
\]
By collecting the previous estimates and taking $N$ large enough we obtain the bound \eqref{imperiale}. We deduce the following lower bound for the Sobolev norm of $z(T)$
\[
\| z(T)\|_s^2\geq \sum_{n\in \Lambda_{N-2}} |n|^{2s} |z_n(T)|^2\geq \frac{\lambda^{-2}}{4} S_{N-2}.
\]
Now we claim that
\begin{equation}\label{bound:initialnorm}
\| z(0)\|_s^2=\| r^{\lambda}(0) \|_s^2\sim \lambda^{-2} S_3.
\end{equation}

Recalling the definition of $r^{\lambda}$ in \eqref{def:rlambda} and using Theorem \ref{thm:orbit} we have
\[
\| r^{\lambda}(0) \|_s^2=\sum_{n\in\Lambda} |n|^{2s} |r^{\lambda}_n(0)|^2\le \lambda^{-2} S_3+\lambda^{-2} \delta^{2\sigma} \sum_{i\neq 3} S_i\le \lambda^{-2} S_3 \left( 1+\delta^{2\sigma} \sum_{i\neq 3} \frac{S_i}{S_3}\right).
\]
By \eqref{bound:gen}, we have that
\[
\delta^{2\sigma} \sum_{i\neq 3} \frac{S_i}{S_3}\le \delta^{2\sigma} C\,(N-1)\,e^{s N}
\]
for some universal constant $C>0$. Since $\delta=e^{-\gamma N}$ with a $\gamma$ large to be fixed (see Theorem \ref{thm:orbit}), we can take $\gamma$ such that $2\gamma\sigma>s$. Then
%We choose {$\gamma=s \tilde{\gamma}$(why you put $s-1?$)} with $\tilde{\gamma}$ large enough such that
\[
\lambda^{-2} S_3 \left( 1+\delta^{2\sigma} \sum_{i\neq 3} \frac{S_i}{S_3}\right)\sim \lambda^{-2} S_3.
\]

This proves the claim \eqref{bound:initialnorm}. To obtain \eqref{ratio} it is enough to use Theorem \ref{thm:gen}, indeed
\[
\frac{\| z(T) \|_s^2}{\| z(0) \|_s^2}\gtrsim\frac{S_{N-2}}{S_3}\gtrsim \,2^{(s-1)(N-4)}\,\omega^{-2s}.
\]
	%The proof of \eqref{ratio} follows by Lemma $7.1$ in \cite{GG21}.
\end{proof}

\noindent\textbf{Conclusion of the proof of Theorem \ref{thm:weak}}
Fix $s>1$, $\cK> 0$ large enough and
\[
\psi(q)=\frac{1}{q\,\log(q)}.
\]
With this choice of $\psi$, by Khinchin theorem \cite{Kinch}
the set of frequencies $\omega$ which are $\psi$-approximable has full Lebesgue measure. We consider $N>0$ such that
\begin{equation}\label{bobo}
 2^{(s-1) (N-6)}\,\omega^{-2s}= {\cK^2} \qquad \Rightarrow \qquad N= \frac{2}{(s-1) \log(2)} \log(\cK \omega^s)+6.
\end{equation}
Note that we can enlarge $N$ by taking $\cK$ larger.
Set {
\[
\lambda=2^{sN}
\]} 
{as in the assumptions of Proposition \ref{prop:approxarg}.}
We can apply Lemma \ref{lem:ratio} which ensures the existence of a solution $z(t)$ of the Hamiltonian system \eqref{calH} satisfying the norm explosion property \eqref{ratio}. By using the mass conservation, the function $u=e^{2 \mathrm{i} \mathcal{M} t} z$ is a solution of the NLS equation \eqref{NLS} and exhibits the norm explosion \eqref{weakI}. Now we show the bound on the $L^2$ norm. We have
\[
\| u(0) \|_{L^2}^2\le \mathtt{C} \lambda^{-2} 2^{N},
\]
for some universal constant $\mathtt{C}>0$. We have that $\mathtt{C} \lambda^{-2} 2^{N}\le \cK^{-1}$ if, for instance,
\[
\cK\geq 2^{6(1-s)} \omega^{-2s} \mathtt{C}.
\]
Then it is sufficient to take $\cK$ large enough to ensure the desired bound on the $L^2$  norm.

Now we prove the estimate on the diffusion time $T$ in  \eqref{weakI}. By the choice of $\lambda$ and \eqref{bobo} we have that, for $\cK>0$ large enough,
\[
\lambda^2=2^{2 s N}=2^{12 s} (\cK\,\omega^s)^{\frac{4s}{s-1}}.
\]
 By \eqref{bound:timeT0} and \eqref{bobo} $T_0\le 2^{4 sN}=\lambda^4$.
By the definition of $T$ in \eqref{def:timeT} and taking $\cK>0$ large enough we have
\[
T=\lambda^2 T_0\le (\cK \omega^{s})^c
\]
for some universal constant $c>0$.

%We first observe that \eqref{bobo} implies that
%\[
%\cK^{-2}\geq \omega^{2s} 2^{-(s-1) (N-6)}=2^N \lambda^{-2}.
%\]
%
%
% comes from the bound on the time $T_0$ \eqref{bound:timeT0}, the rescaling \eqref{def:timeT} with $\lambda\sim 2^{N(1+s)}$ and the fact that $N\sim {\log(K\,\omega^s)}/{(s-1)}$.
%
\subsection{Proof of Theorem \ref{thm:strong}}\label{sec:thm2}

Fix $s>1$, $\tau>2s$, $\mathcal{K}>0$ large enough and $\mu>0$ small enough.
The proof of Theorem \ref{thm:strong} follows the proof of Theorem \ref{thm:weak} until Section \ref{sec:approxarg}.

The approximation argument here is slightly different and it is given by the following.
\begin{proposition}\label{prop:approxarg2}
	%Recall $R$ in \eqref{def:R}, $\gamma$ in Theorem \ref{thm:orbit} and \eqref{set:qlambda}.\\ 
	Fix $\epsilon>0$ small enough. Then, for $\lambda$ and $q$ satisfying the assumption \eqref{assumption} of Lemma \ref{lem:L1},
\begin{equation}\label{cond:final2}
\lambda\geq \exp(5^N), \qquad \vartheta  \le \lambda^{-2(1+\epsilon)}, \qquad \mathtt{L}\geq 1
\end{equation}
	and for $N>0$ large enough
	the following holds. If $r(t)$ is a solution of \eqref{vecfrot} such that
	\begin{equation}\label{assump:initial2}
		\| r(0)-r^{\lambda}(0) \|_{\ell^1}\le \lambda^{-(1+2\epsilon)}
	\end{equation}
	then
	\begin{equation}\label{bound:close2}
		\sup_{t\in [0, T]}\| r(t)-r^{\lambda}(t) \|_{\ell^1}\le \lambda^{-(1+\epsilon)},
	\end{equation}
	where $T$ is the time in \eqref{def:timeT}.
\end{proposition}

The proof is the same of Proposition $2$ in \cite{GG21} and the presence of the constant $\mathtt{L}$ does not affect it. Actually one could consider $\mathtt{L}=1$, since we shall not exploit the size of $\mathtt{L}$ to prove Theorem \ref{thm:strong}.

%\red{\begin{remark}{($\omega=1$)}
%In the square case one has just to ignore the second condition in \eqref{cond:final2}.
%\end{remark}}

 We consider frequencies $\omega$ which are $\psi$--approximable with 
\begin{equation}\label{psitau}
\psi(q)=\frac{\tc}{q^{1+\tau}}, \qquad \tc\geq 1
\end{equation}
and satisfy
\[
 \left|\omega-\frac{p}{q}\right|\geq \frac{1}{q^{1+ \log q}}
\]
for $q$ large enough and any $p\in\mathbb{N}$.
Following \cite{GG21} one can show that given a sequence of convergents $(p_n, q_n)$ of $\omega$
we have
\begin{equation}\label{bound:qn+1}
q_n^{1+\tau}-q_n\le q_{n+1}\le q_n^{\log q_n}.
\end{equation}
Since $\lim_{n\to +\infty} q_n= +\infty$ then, for some fixed $n>0$, there exists $\mathtt{a}, \tilde{\kappa}, \kappa\gg 1$ such that
\[
q_n\in \left[\exp(\tilde{\kappa}\mathtt{a}^N), \exp(\kappa \mathtt{a}^N)\right].
\]
Then, by \eqref{bound:qn+1}, we have that for all $m>0$
\begin{equation}\label{bound:allq2}
q_{n+m}\in \left[\exp(\tilde{\kappa}\mathtt{a}^N), \exp( \kappa^m \mathtt{a}^{mN})\right].
\end{equation}
We shall use this fact to control the upper bound of the instability time.
 We consider $N$ such that
\[
2^{(s-1) (N-6)} \omega^{-2s}\sim \frac{\mathcal{K}^2}{\mu^2} \qquad \Rightarrow \qquad N\sim \frac{\log(\mathcal{K}\omega^s \mu^{-1})}{s-1}.
\]
 We want to impose that the Sobolev norm of the initial datum is  of order $\mu$. 
{ By \eqref{bound:initialnorm} we need to ask that
\begin{equation}\label{cond:roma}
\| z(0)\|_s^2\sim \lambda^{-2} S_3\sim \mu^{2}.
\end{equation}
%We observe that we could consider instead of the set $\Lambda$
%\[
%\Lambda':=\{ \beta n : n\in \Lambda \}
%\]
%for any $\beta\in \mathbb{N}$. Since this rescaling is isotropic it does not affect the bounds in Lemma \ref{lem:U0}. Let us call $\rho$ and $\rho'$ the solutions of the NLS equation that approximate respectively the orbits of the quasi-resonant model on $\Lambda$ and $\Lambda'$. Then $\| \rho(0) \|_s\sim \beta^s \| \rho'(0) \|_s$.
}
By \eqref{def:R} we deduce that
\[
C^{-2s}  2^{N-1}\,\,q^{2 s}\,R^{2s}\le S_3\le C^{2s}  2^{N-1}\,3^{2Ns}\,\omega^{2s}\,q^{2 s}\,R^{2s}.
\]
Thus we can take $\lambda^2\sim \mu^{-2} S_3$, which satisfies
\begin{equation}\label{lowuplambda}
C^{-s} 2^{N/2} q^s  R^s {\mu^{-1}}\lesssim \lambda\lesssim C^s 2^{N/2} 3^{Ns} \omega^s q^s R^s \,\mu^{-1}.
\end{equation}
We note that we can enlarge the above interval by choosing $N$ much larger (hence $\cK$ (or $\mu$) much larger (smaller) ). 

We need to prove that there exist $\lambda$ and $q$ such that: (i) the bounds \eqref{lowuplambda} hold, (ii) the assumptions of Proposition \ref{prop:approxarg2} are satisfied.

Recalling that $R\geq \exp(\alpha^N)$ with $\alpha>0$ large (see Theorem \ref{thm:gen}) it is easy to see that $\lambda$ in \eqref{lowuplambda} satisfies the first inequality in \eqref{cond:final2}, provided $N$ is large enough. By the choice of $\psi$ in \eqref{psitau}, the second inequality in \eqref{cond:final2} is equivalent to 
\[
\begin{cases}
q^{\tau}\gtrsim \lambda^{2(1+\epsilon)} \omega^3 R^2 3^{2N},\\
q^{s_0}\gtrsim \lambda^{2(1+\epsilon)} 4\,R^{-s_0}. 
\end{cases}
\]
We remark that $\epsilon$ given in Proposition \ref{prop:approxarg2} can be considered arbitrarily small.
The above inequalities are compatible with \eqref{lowuplambda} if
\begin{equation}\label{cond:compat}
\begin{cases}
q^{\tau-2s(1+\epsilon)}\gtrsim \omega^3  3^{2N} 2^{N(1+\epsilon)}  R^{2+2s(1+\epsilon)} \mu^{-2(1+\epsilon)},\\
(q R)^{s_0-2s (1+\epsilon)} \gtrsim 2^{N(1+\epsilon)} \mu^{-2(1+\epsilon)}.
\end{cases}
\end{equation}
Let us call 
\[
\nu:=\frac{\tau}{2(1+\epsilon)}-s, \qquad \nu_{*}:=\frac{s_0}{2(1+\epsilon)}-s.
\]
Since $\tau, s_0>2s$ and $\epsilon>0$ is arbitrarily small we have that $\nu, \nu_*>0$. {Therefore if
\begin{equation}\label{explosion}
 \tilde{\kappa}\geq \tilde{C} \frac{s(1+\tilde{\eta})}{\tau-2 s} \qquad \mathrm{for\,\,some\,\,universal\,\,constant}\,\,\tilde{C}>0,
\end{equation}
where $\tilde{\eta}$ is the constant introduced in Theorem \ref{thm:gen},
we have that, for all $q\geq \exp(\tilde{\kappa} \alpha^N)$ and $N$ large enough,
%ho messo il 2 anche se bastava > perché cosi sul tempo T mi viene la costante delta^-1, che è quasi come dire (tau-2s)^-1, come è nel teorema
\begin{equation}\label{lowbound:q}
\begin{cases}
q^{\nu}\gtrsim \omega^{\frac{3}{2(1+\epsilon)}} 3^{\frac{N}{1+\epsilon}} 2^{N/2}  R^{s+\tfrac{1}{1+\epsilon}}\,{\mu^{-1}},\\
(q R)^{\nu_*}\gtrsim 2^{N/2}\,\mu^{-1},
\end{cases}
\end{equation}
which implies the compatibility conditions \eqref{cond:compat}. By the discussion above (see \eqref{bound:allq2}) there exist $\kappa> \tilde{\kappa}$ and $\sigma> 1$ large enough such that the interval $[\exp(\tilde{\kappa} \alpha^N), \exp({\kappa} \alpha^{\sigma N})]$ contains a convergent $q\in\mathbb{N}$ that satisfies \eqref{lowbound:q}.}
Then, one can choose   $\lambda$ in the interval \eqref{lowuplambda} so that \eqref{cond:roma} holds true. That is, 
\begin{equation}\label{caldo}
{C}_*^{-1}\mu \le \| z(0) \|_s\le {C}_* \mu
\end{equation}
for some ${C}_*>0$ independent of $\mu$. Moreover, we have proved that
the chosen $\lambda$ and $q$ satisfy also the assumptions of Proposition \ref{prop:approxarg2}, which we can apply.
% The \eqref{lowuplambda} implies that , which is \eqref{cond:roma}. 
Lemma \ref{lem:ratio} (replacing in the statement Proposition \ref{prop:approxarg} with Proposition \ref{prop:approxarg2}) ensures the growth of Sobolev norms\footnote{The proof follows by Lemma $7.1$ in \cite{GG21}}. Together with \eqref{caldo}, this implies \eqref{bound:uppermu}.

%We are left to prove that such choice of $\lambda$ allows to apply Proposition \ref{prop:approxarg}.
%It is easy to see that by taking $\beta$ large enough $\lambda$ satisfies the first inequality in \eqref{cond:final} of Proposition \ref{prop:approxarg}. 
%Now we see that, thanks to the assumption $\tau>2s$, the second inequality in \eqref{cond:final} is satisfied and then we can apply Proposition \ref{prop:approxarg}. By replacing $\psi$ as in \eqref{psitau} this condition becomes
%\[
%\mathtt{c}^{-1} q^{\tau}\geq  N^{7} 72^{N} \lambda^{2(1+\epsilon)} R^{2}.
%\]
%The above condition is satisfied for ... if
%\[
%\mathtt{a}^N \tau-2(1+\epsilon) s \beta^N-4(1+\eta) \alpha^N>0.
%\]
%This is possible if
%\[
%\mathtt{a}^N (\tau-2(1+\epsilon) s \kappa)>4 (1+\eta) (1+s(1+\epsilon)) \alpha^N
%\]
%and then
%\[
%\tau>2(1+\epsilon) s \kappa.
%\]
%Then \eqref{cond:roma} is equivalent to
%\begin{equation}\label{cond:roma2}
%\lambda\gtrsim {\sqrt{2}^N 3^{Ns}\,\,R^s \mu^{-1} q^s}.
%\end{equation}
%Let us set
%\[
%\lambda=C\,q^s\,(\sqrt{2}^N 3^{Ns}\,\,R^s \mu^{-1})
%\]
%for some $C\gg 1$. The condition \eqref{cond:final} becomes
%\[
%\tc \,q^{s-\tau}=q^{1+s}\psi(q)\le N^{-7} 72^{-N}\,2^{-N}\,3^{-2Ns}\,R^{-2-2s}\, \mu^2.
%\]

%\medskip
%
%Since we consider $\tau>2s$, there exists $q$ satisfying
%\[
% \exp\left({\frac{A_1^N}{\tau-s}}\right)\leq q\leq \exp\left({\frac{A_2^N}{\tau-s}}\right),
%\]
%for some $1\leq A_1<A_2$, which depend on $s$, for which the condition \eqref{cond:final} holds. Clearly $\lim_{\tau\to s^+} q=\infty$. Hence the time $T$ in \eqref{def:timeT} explodes as $\tau\to s^+$.

The upper bound \eqref{time2} on the time $T$ in \eqref{def:timeT} is obtained by using the bounds \eqref{lowuplambda}, the fact that $q\le \exp(\kappa\alpha^{\sigma N})$ and the upper bound of $R$ given in Theorem \ref{thm:gen}. 
%We remark that by \eqref{explosion} if $\tau\to 2s^+$ the constant $\tilde{\kappa}$, and so $\kappa$, tends to $+\infty$. 
%This means that we lose completely the control on the time $T$ when $\tau$ approaches the lower bound $2s$.

\bibliographystyle{plain}
\bibliography{biblio}

\end{document}